\documentclass[article]{amsart}

\usepackage{amssymb,amsfonts,amsmath,amsthm}
\usepackage[all,arc]{xy}
\usepackage{enumerate}
\usepackage{mathrsfs}
\usepackage[toc,page]{appendix}
\usepackage[left=3cm, right=3cm, bottom=3cm]{geometry}
\usepackage{tabularx}
\usepackage{url}
\usepackage{color}
\usepackage{tikz-cd}
\usepackage{mathdots}

\newtheorem{thm}{Theorem}[section]

\newtheorem*{thm*}{Theorem}
\newtheorem*{cor*}{Corollary}
\newtheorem*{prop*}{Proposition}
\newtheorem{cor}[thm]{Corollary}
\newtheorem{prop}[thm]{Proposition}
\newtheorem{lem}[thm]{Lemma}

\theoremstyle{definition}
\newtheorem{defn}[thm]{Definition}

\theoremstyle{remark}
\newtheorem{rem}[thm]{Remark}

\newtheorem*{idea*}{Idea}
\newtheorem*{ntn*}{Notation}

\makeatletter
\let\c@equation\c@thm
\makeatother
\numberwithin{thm}{section}
\numberwithin{equation}{section}

\title[Beauville-Narasimhan-Ramanan correspondence for twisted Higgs $V$-bundles]{The Beauville-Narasimhan-Ramanan correspondence for twisted Higgs $V$-bundles and components of parabolic $\text{Sp}(2n,\mathbb{R})$-Higgs moduli spaces}

\author{Georgios Kydonakis, Hao Sun and Lutian Zhao}

\begin{document}
\maketitle
\begin{abstract}
We generalize the classical Beauville-Narasimhan-Ramanan correspondence to the case of parabolic Higgs bundles with regular singularities and Higgs $V$-bundles. Using this correspondence along with Bott-Morse theoretic techniques we provide an exact component count for moduli spaces of maximal parabolic $\text{Sp}\left( 2n,\mathbb{R} \right)$-Higgs bundles with fixed parabolic structure.
\end{abstract}
\flushbottom
\section{Introduction}

Let $X$ be a smooth, irreducible, projective curve of genus $g\ge 2$ over $\mathbb{C}$ and let $\eta \in \mathsf{\mathcal{A}}=\underset{i=1}{\overset{n}{\mathop \oplus }}\,{{H}^{0}}\left( X,\mathcal{L}^i \right)$, where $\mathcal{L}$ denotes an arbitrary holomorphic line bundle over $X$. Similar to ordinary Higgs bundles \cite{Hit87b}, the $\mathcal{L}$-twisted Higgs bundles have a Hitchin map, which sends a Higgs bundle $\left( E, \Phi \right)$ to the coefficients of the characteristic polynomial of the Higgs field $\Phi$; we refer the reader to \cite {BGL,BiRa,GR,GaRa,Ma} for primary reference on twisted Higgs bundles.

The Beauville-Narasimhan-Ramanan correspondence (Proposition 3.6 in \cite{BNR}) in its simplest form provides a bijective correspondence between isomorphism classes of twisted Higgs bundles over $X$ and line bundles over the spectral curve ${{X}_{\eta }}$, defined by the point $\eta \in \mathsf{\mathcal{A}}$. Under this correspondence, fibers of the Hitchin fibration are Jacobians of ${{X}_{\eta }}$ (see the survey \cite{Dal} for further details on the spectral correspondence).

Considering a set $D=\left\{ {{x}_{1}},\ldots {{x}_{s}} \right\}$ of finitely many points on $X$, a version of the correspondence for parabolic Higgs bundles with Higgs fields having (possibly) irregular singularities at the points in $D$ was proven by S. Szab\'{o} in \cite{Sza}. On the other hand, for local systems on a smooth Deligne-Mumford stack, the Beauville-Narasimhan-Ramanan correspondence provides the existence of a canonical morphism from the stack of flat connections ${{\mathsf{\mathcal{M}}}_{dR}}$ to the Frobenius twist ${{\mathsf{\mathcal{A}}}^{(1)}}$, called the twisted Hitchin morphism (see \cite{Gro}).

In this article, the focus is on parabolic Higgs bundles with regular (tame) singularities and we prove the following version of this correspondence:

\begin{thm*}{\bf{\ref{313}}}
Let $X$ be a closed Riemann surface, $D=\{x_1,...,x_s\}$ a fixed set of $s$-many points on $X$ and let $m$ be a fixed positive integer. Denote by   $K=\Omega _{X}^{1}$ the canonical line bundle over $X$  and consider $K\left( D \right):=K\otimes {{\mathsf{\mathcal{O}}}_{X}}\left( D \right)$. Fix a parabolic structure $\alpha$ for a rank $n$ parabolic bundle over $X$ and a tuple of sections $\eta=(\eta_i)$, where $\eta_i$ is a section of $K(D)^i$ for $1 \leq i \leq n$. Assume that the surface $X_\eta$ is non-singular and the intersection of the branch points $B$ and the given divisor $D$ is empty. For a fixed order of the pre-image $\pi^{-1}(x)=\{\widetilde{x}_1,...,\widetilde{x}_n\}$ of each $x \in D$, there is a bijective correspondence between isomorphism classes of strictly compatible parabolic line bundles $(L,\widetilde{\alpha})$ on $X_\eta$ and isomorphism classes of pairs $(E, \Phi)$, where $E$ is a parabolic bundle of rank $n$ with parabolic structure $\alpha$ and $\Phi:E \rightarrow E \otimes K(D)$ a parabolic Higgs field with characteristic coefficients $\eta_i$.
\end{thm*}

In the case when the weights in the parabolic structure are rational numbers, there is a bijection between moduli of semistable parabolic Higgs bundles over $X$ and semistable Higgs bundles over an orbifold $M$ (called a $V$-surface) with underlying manifold the Riemann surface $X$. In the construction of the orbifold surface the denominator of the weights is describing the cyclic group action around the marked points of $D$ (see \cite{FuSt,KySZ,NaSt} for the detailed construction). For our purposes, we provide here an $\mathcal{L}$-twisted version of this correspondence (Proposition \ref{3030}), for $\mathcal{L}$ a line $V$-bundle over $M$. Let $\eta =\left( {{\eta }_{i}} \right)\in \mathsf{\mathcal{A}}$ be a set of sections of ${{\mathsf{\mathcal{L}}}^{i}}$, for $1\le i\le n$. We construct a spectral covering $\pi: {{M}_{\eta }}\to M$ of a $V$-surface $M$ by pulling back the spectral cover ${{X}_{\eta }}\to X$ and defining an atlas for the $V$-surface ${{M}_{\eta }}$ over $\left({X}_{\eta }, {\pi}^{-1}\left( D\right)\right)$. The following corollary is then implied from this version of the correspondence:

\begin{cor*}{\bf{\ref{315}}}
Assume that the underlying surface of $M_\eta$ is nonsingular and the intersection of the set of branch points $B$ and the divisor $D$ is empty. Then there is a bijective correspondence between isomorphism classes of strictly compatible line $V$-bundles $L$ on $M_\eta$ and isomorphism classes of pairs $(E, \Phi)$, where $E$ is a $V$-bundle of rank $n$ over $M$ and $\Phi:E \rightarrow E \otimes \mathcal{L}$ a homomorphism with characteristic coefficients $\eta_i$. For $\rho : M_\eta \rightarrow M$, it is $\rho_*(L)=E$.
\end{cor*}

In order to provide a complete description of the Beauville-Narasimhan-Ramanan correspondence and the correspondence among parabolic Higgs bundles, Higgs $V$-bundles and Higgs bundles over a root stack, we have one-to-one correspondences between the following six categories:
\begin{center}
\begin{tikzcd}[column sep=12ex]	
	\text{$\mathcal{L}$-twisted Higgs $V$-bundles over $M$ } \arrow[r,"\text{ Corollary \ref{315}}"] \arrow[d,"\text{ Proposition \ref{3030}}"] & \text{ line $V$-bundles over $M_\eta$} \arrow[d,"\text{ Proposition \ref{3030}}"]\\

	 \text{ parabolic $\mathbb{L}$-twisted Higgs bundles over $X$} \arrow[r,"\text{ Theorem \ref{313}}"] \arrow[d,"\text{ Proposition \ref{318}}"] & \text{parabolic line bundles over $X_\eta$} \arrow[d,"\text{ Proposition \ref{318}}"] \\
	
	\text{ $\mathfrak{L}$-twisted Higgs bundles over stack $\mathfrak{X}$} \arrow[r,"\text{ Corollary \ref{319}}"] & \text{ line bundles over $\mathfrak{X}_\eta$}.
\end{tikzcd}	
\end{center}

For (non-parabolic) $G_{\mathbb{C}}$-Higgs bundles, the construction of spectral curves from \cite{BNR} implies that each connected component of a generic fiber of the Hitchin map is an abelian variety (see for instance \cite{BaScha} for a detailed description). Analogously to this approach, we may describe the regular fibers of the parabolic Hitchin map for the moduli space $\mathcal{M}_{par}(X,D,\alpha, n)$ of rank $n$ parabolic Higgs bundles with a given parabolic structure $\alpha$. This fibration sends a Higgs bundle $(E, \Phi)$ to the coefficients of the characteristic polynomial of $\Phi$, while in the case when the Higgs field is assumed to be strongly parabolic, then the eigenvalues of $\Phi$ all vanish at the points in the divisor $D$. The Hitchin map for parabolic Higgs bundles was first defined in \cite{LoMa} for non-strongly parabolic Higgs fields $\Phi$ and has been studied for strongly parabolic $\Phi$ in \cite{GoLo}.

We are interested in this article, however, in considering weights in the parabolic structure as rational numbers with denominator $m$ and study the fibration through the prism of the aforementioned correspondence to Higgs $V$-bundles over $M$. For a positive integer $m \ge 2$, denote by ${{\mathsf{\mathcal{M}}}_{par}}\left( X,D,\alpha,n ,m \right)$ the moduli space of parabolic Higgs bundles $\left( E,\Phi  \right)$ such that the weights in the system $\alpha $ can be written as rational numbers with denominator $m$; note that ${{\mathsf{\mathcal{M}}}_{par}}\left( X,D,\alpha,n ,m \right)\subset {{\mathsf{\mathcal{M}}}_{par}}\left( X,D,\alpha,n  \right)$. Similarly, denote by ${{\mathsf{\mathcal{M}}}_{par}}\left( X,D,n \right)$ the moduli space of rank $n$ poly-stable parabolic Higgs bundles $\left( E,\Phi  \right)$ with arbitrary parabolic structure, and ${{\mathsf{\mathcal{M}}}_{par}}\left( X,D,n,m \right)$ the moduli space of rank $n$ poly-stable parabolic Higgs bundles $\left( E,\Phi  \right)$ equipped with a parabolic structure in which the weights can be written as rational numbers with denominator $m$.

We deduce the following:
\begin{prop*}{\bf{\ref{3170}}}
The fiber of the parabolic Hitchin map $$\mathcal{M}_{par}(X,D,n,m) \rightarrow \bigoplus\limits^n_{i=1} H^0(X,K(D)^i)$$ over a regular point $\eta \in \bigoplus\limits^n_{i=1} H^0(X,K(D)^i)$ is isomorphic to the $V$-Picard group of the spectral covering $M_\eta$ of $M$, where $M$ is uniquely determined by the data $(X,D,m)$.
\end{prop*}

In the special case of parabolic Higgs bundles with rank 2 and rational weights with denominator 2, the parabolic Hitchin fibration of a regular point $\eta \in H^0(X,K(D)^2)$ is precisely the Prym variety
\begin{align*}
\text{Prym}(M_\eta,M)=\{L \in \text{Pic}_V(M_\eta)| \tau^*L =L^{-1}\}
\end{align*}
of the spectral covering $M_\eta \rightarrow M$, where $M$ and $M_\eta$ are the corresponding $V$-surfaces of $X$ and $X_\eta$, and $\tau:M_\eta \rightarrow M_\eta$ is the involution of $M_\eta$.
Namely, we show:

\begin{prop*}{\bf{\ref{317}}}
The fiber of the parabolic Hitchin map
\begin{align*}
\mathcal{M}_{par}(X,D,2,2) \rightarrow H^0(X,K(D)^2)
\end{align*}
is a finite number of copies of the Prym variety $\text{Prym}(M_\eta,M)$. This number of copies only depends on the parabolic structure $\alpha$.
\end{prop*}

We next focus on the problem of providing an exact component count for moduli spaces of maximal parabolic $\text{Sp}\left( 2n,\mathbb{R} \right)$-Higgs bundles $\left( E,\Phi  \right)$ over $\left( X,D \right)$. In this direction, we employ Bott-Morse theoretic techniques, the origins of which are traced back to N. Hitchin's seminal work \cite{Hit} and the work of P. Gothen \cite{Goth}, studying the topology of moduli spaces of Higgs bundles in the absence of a parabolic structure. In the case of parabolic Higgs bundles these techniques were first developed by M. Logares in \cite{Log} for the moduli space of parabolic $\text{U}(p,q)$-Higgs bundles and have been used in \cite{GaGoMu} in computing the Betti numbers of moduli of parabolic $\text{GL}\left( 3,\mathbb{C} \right)$- and $\text{SL}\left( 3,\mathbb{C} \right)$-Higgs bundles. Furthermore, in \cite{Ra} and \cite{RaSu} Bott-Morse theory for twisted Higgs bundles in the non-parabolic/non-orbifold context was implemented, with particular attention to the minima of a positive functional and pullback diagrams for critical points. In parallel to the aforementioned work, we study the spaces of minima of the positive functional
	\[f\left( \left[ \left( E,\Phi  \right) \right] \right)=\frac{1}{2}{{\left\| \Phi  \right\|}^{2}}.\]
This map is non-negative and proper, thus a subspace ${\mathsf{\mathcal{N}}}$ of this moduli space is connected, if the subspace of local minima of $f$ in ${\mathsf{\mathcal{N}}}$ is. The subspaces of local minima coincide with the critical subvarieties, thus correspond to variations of Hodge structures. We compute the Morse indices of $f$ and show that particular subvarieties of the moduli space defined by appropriate topological invariants are connected. The method is parallel to the non-parabolic version studied by P. Gothen in \cite{Goth}. We prove accordingly:

\begin{thm*}{\bf{\ref{507}}}
The moduli space ${{\mathsf{\mathcal{M}}}_{par}^{\text{max}}}\left( \text{Sp}\left( 4,\mathbb{R} \right) \right)$ of parabolic maximal $\text{Sp}\left( 4,\mathbb{R} \right)$-Higgs bundles with all weights rational having denominator 2 over a compact Riemann surface $X$ of genus $g$ with a divisor of $s$-many distinct points on $X$, such that $2g-2+s>0$, has $(2^s+1)2^{2g+s-1}+(2g-3+s)2^{s}$ many connected components.
\end{thm*}

\begin{thm*}{\bf{\ref{510}}}
For $n \ge 3$, the moduli space ${{\mathsf{\mathcal{M}}}_{par}^{\text{max}}}\left( \text{Sp}\left( 2n,\mathbb{R} \right) \right)$  of parabolic maximal $\text{Sp}\left( 2n,\mathbb{R} \right)$-Higgs bundles with all weights rational having denominator 2 over a compact Riemann surface $X$ of genus $g$ with a divisor of $s$-many distinct points on $X$, such that $2g-2+s>0$, has $(2^s+1)2^{2g+s-1}$ many connected components.
\end{thm*}

For points $\left( V,\beta ,\gamma  \right)$ in the moduli space  ${{\mathsf{\mathcal{M}}}_{par}^{\text{max}}}\left( \text{Sp}\left( 2n,\mathbb{R} \right) \right)$ for any $n \ge 2$, we further fix the parabolic structure on the parabolic bundle $V$ and study the connected component count problem for different values of rational weights. Among other results, we verify the prediction made in \cite{KySZ} on the number of components when the flag $\hat{ \alpha}$ on $V$ is trivial with weight $\frac{1}{2}$:

\begin{cor*}{\bf{\ref{801}}}
The number of connected components of the moduli space ${{\mathsf{\mathcal{M}}}_{par}^{\text{max}}}\left( \text{Sp}\left( 4,\mathbb{R} \right), \hat{ \alpha} \right)$ with fixed trivial filtration and weight $\frac{1}{2}$ is
\begin{align*}
(2^s+1)2^{2g+s-1}+(2g-2+s)-2^s.
\end{align*}
\end{cor*}

In the final section of this paper, we generalize our approach to count the connected components of $\mathcal{M}_{par,(m_1,\ldots,m_s)}^{\text{max}}(\text{Sp}(2n,\mathbb{R}))$, where $\mathcal{M}_{par,(m_1,\ldots,m_s)}^{\text{max}}(\text{Sp}(2n,\mathbb{R}))$ denotes the moduli space of polystable maximal $\text{Sp}(2n,\mathbb{R})$-Higgs bundles with weight type of $(m_i)_{1\le i\le s}$; by the latter, we mean that the weight $\alpha_{i,j}$ at each point $x_i\in D$ in the parabolic structure of the bundle is an integral multiple of $\frac{1}{m_i}$ for $1 \leq i \leq s$ and $1 \leq j \leq n$. When $m_1= \dots =m_s=2$, $\mathcal{M}_{par,(2,\ldots,2)}^{\text{max}}(\text{Sp}(2n,\mathbb{R}))$ is exactly the moduli space ${{\mathsf{\mathcal{M}}}_{par}^{\text{max}}}\left( \text{Sp}\left( 2n,\mathbb{R} \right) \right)$ considered above for $n \ge 2$. The following results describe a count of the number of connected components:

\begin{thm*}{\bf{\ref{903}}}
Let $X$ be a compact Riemann surface of genus $g$ with a divisor $D$ of $s$-many distinct points on $X$, such that $2g-2+s>0$. Then, the moduli space of polystable maximal parabolic $\text{Sp}(4,\mathbb{R})$-Higgs bundles on $(X,D)$ with weight type $(m_i)_{1\le i\le s}$ has $$(2^{s_0}+1)2^{2g+s_0-1}-2^{s_0}+(2g-2+s)\prod_{i=1}^s m_i$$ connected components, where $s_0$ is the number of the even $m_{i}$ in the collection $(m_i)_{1\le i\le s}$.
\end{thm*}

\begin{thm*}{\bf{\ref{904}}}
Let $X$ be a compact Riemann surface of genus $g$ with a divisor $D$ of $s$-many distinct points on $X$, such that $2g-2+s>0$. Then, the moduli space of polystable maximal parabolic $\text{Sp}(2n,\mathbb{R})$-Higgs bundles $(n \geq 3)$ on $(X,D)$ with weight type $(m_i)_{1\le i\le s}$ has $$(2^{s_0}+1)2^{2g+s_0-1}$$ connected components, where $s_0$ is the number of the even $m_{i}$ in the collection $(m_i)_{1\le i\le s}$.
\end{thm*}

Let $\alpha_{ij}=\frac{k_{ij}}{m_i}$ be the weights over the puncture $x_{i}\in D$ with the same denominator $m_i$, where $k_{ij}$ is a positive integer, for $1 \leq i \leq s$ and $1 \leq j \leq n$. We call the parabolic structure $\alpha$ non-reduced, if there is at least one enumerator $k_{ij}$ such that $k_{ij}$ and $m_i$ are coprime, for any $1 \leq i \leq s$. Using the same arguments as in the preceding two theorems, we finally obtain:

\begin{cor*}{\bf{\ref{905}}}
Let $X$ be a compact Riemann surface of genus $g$ with a divisor $D$ of $s$-many distinct points on $X$, such that $2g-2+s>0$. Then, the moduli space of polystable maximal parabolic $\text{Sp}(4,\mathbb{R})$-Higgs bundles on $(X,D)$ with any non-reduced parabolic structure $\alpha$ has $$(2^{s_0}+1)2^{2g+s_0-1}-2^{s_0}+(2g-2+s)$$ connected components, where $s_0$ is the number of the even $m_{i}$ in the collection $(m_i)_{1\le i\le s}$.
\end{cor*}

\section{Parabolic $\text{Sp}\left( 2n,\mathbb{R} \right)$-Higgs bundles and their moduli}
In this section, we review basic facts about moduli spaces of parabolic Higgs bundles and set notation.

\subsection{Parabolic vector bundles}

Let $X$ be a closed, connected, smooth Riemann surface of non-negative genus  $g$ together with a finite collection of distinct points ${{x}_{1}},\ldots ,{{x}_{s}}$. Denote by $D$ the effective divisor
\[D={{x}_{1}}+\ldots +{{x}_{s}}\]
with ${{x}_{i}}\ne {{x}_{j}}$, for $i\ne j$. The pair $\left( X,D \right)$ will be kept fixed throughout.

\begin{defn}\label{204}
A \emph{parabolic vector bundle} $E$ of rank $n$ over $\left( X,D \right)$ is a holomorphic vector bundle together with a \emph{parabolic structure} along $D$, that is, a collection of weighted flags on each fiber ${{E}_{x}}$ for $x\in D$:
\[\begin{matrix}
   {{E}_{x}}={{E}_{x,1}}\supset {{E}_{x,2}}\supset \ldots \supset {{E}_{x,r\left( x \right)+1}}=\left\{ 0 \right\}  \\
   0\le {{\alpha }_{1}}\left( x \right)<\ldots <{{\alpha }_{r\left( x \right)}}\left( x \right)<1
\end{matrix}\]
where $r\left( x \right)$ is an integer between $1$ and $n$.\end{defn}

The real numbers ${{\alpha }_{i}}\left( x \right)\in \left[ 0,1 \right)$ for $1\le i\le r\left( x \right)$ are called the \emph{weights} of the subspaces ${{E}_{x}}$ and we usually write $\left( E, \alpha \right)$ to denote a parabolic vector bundle equipped with a parabolic structure determined by a system of weights $\alpha \left( x \right)=\left( {{\alpha }_{1}}\left( x \right),\ldots ,{{\alpha }_{r(x)}}\left( x \right) \right)$ at each $x\in D$; whenever the system of weights is not discussed in the context, we will be omitting the notation  $\alpha$ to ease exposition. Moreover, let ${{k}_{i}}\left( x \right)=\dim\left( {{{E}_{x,i}}}/{{{E}_{x,i+1}}}\; \right)$ denote the \emph{multiplicity} of the weight ${{\alpha }_{i}}\left( x \right)$ and notice that $\sum\limits_{i}{{{k}_{i}}}\left( x \right)=n$. A weighted flag shall be called \emph{full}, if ${{k}_{i}}\left( x \right)=1$ for every $1\le i\le r\left( x \right)$  and every $x\in D$.

We define \emph{parabolic degree} and \emph{parabolic slope} of a parabolic vector bundle over $\left( X,D \right)$ as follows
\[par\deg \left( E \right)=\deg E+\sum\limits_{x\in D}{\sum\limits_{i=1}^{r\left( x \right)}{{{k}_{i}}\left( x \right){{\alpha }_{i}}\left( x \right)}}\]
\[par\mu \left( \text{E} \right)=\frac{\text{pardeg}\left( E \right)}{\text{rk}\left( E \right)}.\]

In the sequel, we will be making frequent use of the notion of subbundle, dual, direct sum and tensor product of parabolic vector bundles. We review these constructions below:

\paragraph{\emph{Parabolic subbundle}.} Let $\left( E,\alpha  \right)$ be a parabolic vector bundle over $\left( X,D \right)$. A \emph{parabolic subbundle} $\left( V,\beta  \right)$ of $\left( E,\alpha  \right)$ is described by a holomorphic vector subbundle $V\subseteq E$ together with a parabolic structure along the divisor $D$ given by the flag
	\[{{V}_{x}}={{V}_{x,1}}\supset {{V}_{x,2}}\supset \ldots {{V}_{x,r\left( x \right)+1}}=\left\{ 0 \right\}\]
for ${{V}_{x,i}}={{V}_{x}}\cap {{E}_{x,i}}$, and a system of weights  $\beta \left( x \right)=\left( {{\beta }_{1}}\left( x \right),\ldots {{\beta }_{r\left( x \right)}}\left( x \right) \right)$ defined by
\begin{align*}
 {{\beta }_{i}}\left( x \right) & ={{\max }_{j}}\left\{ {{\alpha }_{j}}\left( x \right)\left| {{V}_{x}}\cap {{E}_{x,j}}={{V}_{x,i}} \right. \right\}\\
& = {{\max }_{j}}\left\{ {{\alpha }_{j}}\left( x \right)\left| {{V}_{x,i}}\subseteq {{E}_{x,j}} \right. \right\}.
\end{align*}
In other words, the weight for ${{V}_{x,i}}$ is the weight ${{\alpha }_{j}}\left( x \right)$ for which ${{V}_{x,i}}\subseteq {{E}_{x,j}}$ but ${{V}_{x,i}}\not\subseteq{E}_{x,j+1}$.

\paragraph{\emph{Parabolic dual}.} Let $\left( E,\alpha  \right)$ be a parabolic vector bundle over $(X,D)$. Consider the following filtration for each point $x\in D$:
\[ E^\vee_x= E^\vee_{x,1} \supset \dots \supset E^{\vee}_{x,r(x)} \supset 0\]
\[ 1-\alpha_ {r\left( x \right)} (x)< \dots < 1-\alpha_1(x)\]
where $E^\vee_{x,i} := \text{Hom}(E_x/E_{x,r(x)+2-i}, \mathcal{O}(-D)_x)$. It is easy to check that
\[ (E^{\vee})^{\vee}=E \quad \text{and} \quad par\deg(E^\vee)=-par\deg(E).\]
We call $E^\vee$ the \emph{parabolic dual} of $E$.

\paragraph{\emph{Parabolic direct sum}.} Let $\left( E,\alpha  \right)$ and $\left( {E}',{\alpha }' \right)$ be two parabolic vector bundles over the pair $\left( X,D \right)$. Their \emph{parabolic direct sum} $\left( \tilde{E},\tilde{\alpha } \right)$ is defined as the holomorphic vector bundle $\tilde{E}:=E\oplus {E}'$ together with a system of weights $\tilde{\alpha }$ comprised of the ordered collection of the weights in $\alpha $ and ${\alpha }'$ corresponding to the filtration of ${{\tilde{E}}_{x,k}}={{E}_{x,i}}\oplus {{{E}'}_{x,j}}$, where  $i$ (resp.$j$) is the smallest integer such that ${{\tilde{\alpha }}_{k}}\left( x \right)\le {{\alpha }_{i}}\left( x \right)$ (resp. ${{\tilde{\alpha }}_{k}}\left( x \right)\le {{{\alpha }'}_{j}}\left( x \right)$). Under this definition the parabolic degree now satisfies the relation
\[ par\deg(E\oplus {E}')=par\deg(E) + par\deg(E') .\]

\paragraph{\emph{Parabolic tensor product}.}
 Let now $\left( E,\alpha  \right)$ be a parabolic vector vector bundle and $(L,\beta)$ be a parabolic line bundle over $\left( X,D \right)$. Notice that $L$ is a locally free sheaf of rank $1$, thus the filtration of $L$ at a point $x \in D$ is always given by the trivial flag $L_x \supset 0$ with weight $\beta(x)$ for this filtration. The \emph{parabolic tensor product} $E \otimes L$ is defined as the kernel of the following map
\begin{align*}
E \otimes L(D) \twoheadrightarrow \bigoplus_{x \in D} \left( (E_x/E_{x,i_x}) \otimes L(D)_p \right),
\end{align*}
where $i_x= \text{min} \{ r(x)+1, i | \alpha_i(x)+\beta(x) \geq 1 \}$, for $x \in D$. The filtration is given by
\begin{align*}
E_{x,i_x} \otimes L(D)_x \supset \dots \supset E_{x,r(x)} \otimes L(D)_x \supset E_{x,1} \otimes L_x \supset \dots \supset E_{x,i_x-1} \otimes L_x
\end{align*}
with weights
\begin{align*}
\alpha_{i_x}(x)+\beta(x)-1 < \dots < \alpha_{r(x)}(x)+\beta(x)-1 < \alpha_1(x)+\beta(x)< \dots < \alpha_{i(x)-1}(x)+\beta(x).
\end{align*}
This construction can be extended to vector bundles or locally free sheaves of any rank; details can be found in \cite{Yoko2}. The parabolic tensor product of two parabolic vector bundles $E$ and $F$ now satisfies
\[ par\deg(E \otimes F)=\text{rk}(F) par\deg(E) + \text{rk}(E) par\deg(F) .\]

\begin{defn}\label{205}
For a pair of parabolic vector bundles $\left( E, \alpha \right),\left( {E}', \alpha^{\prime } \right)$ over the pair $\left( X,D \right)$  a holomorphic map $f:E\to {E}'$ is called \emph{parabolic} if ${{\alpha }_{i}}\left( x \right)>{{{\alpha }'}_{j}}\left( x \right)$ implies $f\left( {{E}_{x,i}} \right)\subset {{{E}'}_{x,j+1}}$ for every $x\in D$. Furthermore, we call such map \emph{strongly parabolic} if ${{\alpha }_{i}}\left( x \right)\ge {{{\alpha }'}_{j}}\left( x \right)$ implies $f\left( {{E}_{x,i}} \right)\subset {{{E}'}_{x,j+1}}$ for every $x\in D$.
\end{defn}

\begin{defn}\label{207}
A parabolic vector bundle $\left( E,\alpha  \right)$ over $(X,D)$ will be called \emph{stable} (resp. \emph{semistable}), if for every non-trivial proper parabolic subbundle $V\le E$ it is $\text{par}\mu \left( V \right)<\text{par}\mu \left( E \right)$ (resp. $\le $). A parabolic vector bundle is called  \emph{poly-stable}, if it is a direct sum of stable parabolic vector bundles of the same parabolic slope.
\end{defn}

\subsection{Moduli of parabolic Higgs bundles}

\begin{defn}
Let $\left( X,D \right)$ be a pair of a Riemann surface together with a divisor. Denote by   $K=\Omega _{X}^{1}$ the canonical line bundle over $X$  and consider $K\left( D \right):=K\otimes {{\mathsf{\mathcal{O}}}_{X}}\left( D \right)$. A \emph{(strongly) parabolic Higgs bundle} $\left( E,\Phi  \right)$ is a pair of a parabolic vector bundle $E$ and a (strongly) parabolic morphism $\Phi :E\to E\otimes K\left( D \right)$ called a \emph{Higgs field}.
\end{defn}

\begin{rem}
For a strongly parabolic Higgs bundle $\left( E,\Phi  \right)$ the Higgs field $\Phi $ is a meromorphic endomorphism valued 1-form with at most simple poles along the divisor $D$ and with  $\text{Re}{{\text{s}}_{x}}\Phi $ nilpotent for each $x\in D$. In other words, it is satisfied that $\Phi \left( {{E}_{x,i}} \right)\subseteq {{E}_{x,i+1}}\otimes K{{\left( D \right)}_{x}}$ for every $x\in D$.
\end{rem}

\begin{defn}\label{250}
A (strongly) parabolic Higgs bundle $\left( E,\Phi  \right)$ will be called \emph{stable} (resp. \emph{semistable}), if for every non-trivial $\Phi $-invariant parabolic subbundle $V\subseteq E$, it holds that  $par\mu \left( V \right)<par\mu \left( V \right)$ (resp.$\le $); by $\Phi $-invariant it is meant here that $\Phi \left( V \right)\subseteq V\otimes K\left( D \right)$.
\end{defn}

For a given rank $n$ and a given parabolic structure $\alpha$ on the underlying parabolic vector bundle $E$, a moduli space of semistable parabolic Higgs bundles $\left( E,\Phi  \right)$ over $\left( X,D \right)$ was constructed by K. Yokogawa as a complex quasi-projective variety, which is smooth at the stable points (see \cite{Yoko1}, \cite{Yoko2}). This moduli space parameterizes isomorphism classes of semistable parabolic Higgs bundles, where the isomorphism classes are also called the $S$-equivalent classes and are defined from the Jordan-H\"older filtration (see \cite[\S 4]{Yoko1}). We prefer to consider the poly-stable objects in this moduli space. Let ${{\mathsf{\mathcal{M}}}_{par}}\left( X,D,\alpha,n  \right)$ denote the moduli space of rank $n$ poly-stable parabolic Higgs bundles over $(X,D)$ with parabolic structure $\alpha$.

In this article, we wish to restrict to smaller moduli spaces than ${{\mathsf{\mathcal{M}}}_{par}}\left( X,D,\alpha,n  \right)$. For a positive integer $m \ge 2$, denote by ${{\mathsf{\mathcal{M}}}_{par}}\left( X,D,\alpha,n ,m \right)$ the moduli space of parabolic Higgs bundles $\left( E,\Phi  \right)$ such that the weights in the system $\alpha $ can be written as rational numbers with denominator $m$. Clearly, $${{\mathsf{\mathcal{M}}}_{par}}\left( X,D,\alpha,n ,m \right)\subset {{\mathsf{\mathcal{M}}}_{par}}\left( X,D,\alpha,n  \right).$$

Similarly, denote by ${{\mathsf{\mathcal{M}}}_{par}}\left( X,D,n \right)$ the moduli space of rank $n$ poly-stable parabolic Higgs bundles $\left( E,\Phi  \right)$ with arbitrary parabolic structure, and ${{\mathsf{\mathcal{M}}}_{par}}\left( X,D,n,m \right)$ the moduli space of rank $n$ poly-stable parabolic Higgs bundles $\left( E,\Phi  \right)$ equipped with a parabolic structure in which the weights can be written as rational numbers with denominator $m$.

\subsection{Parabolic $\text{Sp}\left( 2n,\mathbb{R} \right)$-Higgs bundles}

In the sequel, we shall be considering parabolic Higgs bundles for structure group $G=\text{Sp}\left( 2n,\mathbb{R} \right)$. A general theory of parabolic $G$-Higgs bundles for a real reductive group $G$ was carried out in  \cite{BiGaRi}, where a Hitchin-Kobayashi correspondence was also established. For our purposes, we will consider here parabolic $\text{Sp}\left( 2n,\mathbb{R} \right)$-Higgs bundles as special parabolic Higgs bundles  $\left( E,\Phi  \right)$; one can check though that the general definition from  \cite{BiGaRi} specializes when $G=\text{Sp}\left( 2n,\mathbb{R} \right)$ to the definition we describe next.

Let $\left( X,D \right)$ and $K\left( D \right)$ as defined in \S 2.1 and \S 2.2. A maximal compact subgroup of $G=\text{Sp}(2n,\mathbb{R})$ is $H \cong \text{U}(n)$ and ${{H}^{\mathbb{C}}}=\text{GL}( n,\mathbb{C})$, thus the parabolic structure is defined on a $\text{GL}\left( n,\mathbb{C} \right)$-principal bundle.

\begin{defn}\label{401}
A \emph{parabolic $\text{Sp}( 2n,\mathbb{R} )$-Higgs bundle} over $\left( X,D \right)$ is defined as a triple $\left( V,\beta ,\gamma  \right)$, where
\begin{itemize}
\item $V$ is a holomorphic rank $n$ parabolic bundle on $X$ such that all of the weights have denominator $2$.
\item The maps  $\beta :{{V}^{\vee }}\to V\otimes K(D)$ and $\gamma :V\to {{V}^{\vee }}\otimes K(D)$ are symmetric parabolic morphisms.
\end{itemize}
\end{defn}

The parabolic structures on $V$ and ${{V}^{\vee }}$ now induce a parabolic structure on the parabolic direct sum $E=V\oplus {{V}^{\vee }}$, for which $par\deg E=0$. The definition of a parabolic $\text{Sp}( 2n,\mathbb{R})$-Higgs bundle on $\left( X,D \right)$ is now reinterpretated as a pair $\left( E,\Phi  \right)$ such that
\begin{align*}
E=V\oplus {{V}^{\vee }}, \quad \Phi =\left( \begin{matrix}
   0 & \beta   \\
   \gamma  & 0  \\
\end{matrix} \right):E\to E\otimes K\left( D \right).
\end{align*}
The stability condition for such pairs $\left( E,\Phi  \right)$ is given in Definition \ref{250}. In the rest of the paper, we shall denote by $\mathcal{M}_{par}(\text{Sp}(2n,\mathbb{R}))$ the moduli space of the triples $(V,\beta,\gamma)$ defined above. Note that in Definition \ref{401} the weights of the parabolic structure for $V$ have denominator $2$. Therefore, for us
\begin{align*}
\mathcal{M}_{par}(\text{Sp}(2n,\mathbb{R})) = \mathcal{M}_{par}(X,D,n,2),
\end{align*}
in the notation of \S 2.2.

Now we consider a special parabolic structure. Let $V$ be a holomorphic rank $m$ vector bundle defined over the compact Riemann surface $X$. We fix a parabolic structure on $V$ as follows. The parabolic structure is defined by a trivial flag ${{V}_{x}}\supset \left\{ 0 \right\}$ and weight $\frac{1}{2}$ for each ${{V}_{x}}$ and $x\in D$. Then the parabolic symmetric power $\text{Sym}^n(V)$ is equipped with the trivial flag and weight $\frac{1}{2}$. Denote by $V^{\vee}$ the parabolic dual of $V$, which is defined as ${{\left( {{V}^{*}} \right)}}\otimes {\mathcal{O}_X(D)^{*}}$. The parabolic symmetric power for the parabolic dual $\text{Sym}^n(V^\vee)$ is defined as the symmetric product of the bundle ${{\left( {{V}^{*}} \right)}}\otimes {\mathcal{O}_X(D)^{*}} $ equipped with a parabolic structure given by the trivial flag and weight $\frac{1}{2}$. Denote this parabolic structure by $\hat{\alpha}$, and define
\begin{align*}
\mathcal{M}_{par}(\text{Sp}(2n,\mathbb{R}),\hat{\alpha}) := \mathcal{M}_{par}(X,D,\hat{\alpha},n,2).
\end{align*}
We will compute the number of connected components of this moduli space (with maximal parabolic degree) in \S 8.

\begin{rem}\label{402}
In the definition of a parabolic $\text{Sp}\left( 2n,\mathbb{R} \right)$-Higgs bundle we could have considered any parabolic flag on the fiber  ${{V}_{x_{i}}}$:
\begin{equation}\tag{2}
\begin{matrix}
   {{V}_{x_{i}}}\supset {{V}_{x_{i},2}}\supset \ldots \supset {{V}_{x_{i},n+1}}=\left\{ 0 \right\}  \\
   0\le {{\alpha }_{1}}\left( x_{i} \right)\le \ldots \le  {{\alpha }_{n}}\left( {{x}_{i}} \right)<1  \\
\end{matrix}
\end{equation}
for each ${{x}_{i}}\in D$. The reason for fixing the trivial flag with weight $\frac{1}{2}$ will become clear later on in \S 4, where we use the data  $\left( X,D,m=2 \right)$  in order to construct the corresponding $V$-surface $M$ and use the topological invariants of Higgs $V$-bundles to study the topology of the moduli space of parabolic Higgs bundles. These invariants will be defined as characteristic classes in $V$-cohomology groups with ${{\mathbb{Z}}_{2}}$-coefficients.
\end{rem}

\section{Bott-Morse Theory on $\mathcal{M}_{par}(\text{Sp}\left( 2n,\mathbb{R} \right))$}

Similarly to the non-parabolic case, there is a natural $\mathbb{C}^*$-action on the moduli space $\mathcal{M}_{par}(\text{Sp}\left( 2n,\mathbb{R} \right))$ given by the map $\lambda \cdot (E,\Phi)=(E,\lambda \Phi)$. For finding solutions to the Hitchin equations, we restrict the action to $S^1 \subset \mathbb{C}^*$. With respect to the complex structure studied by H. Konno \cite{Konn}, this is a Hamiltonian action and the associated moment map is
\begin{equation*}
[(E, \Phi)] \rightarrow -\frac{1}{2} \parallel \Phi \parallel^2=-i\int_X \text{Tr}(\Phi \Phi^*).
\end{equation*}
We choose instead to consider the positive function
\begin{equation*}
f([(E, \Phi)])= \frac{1}{2} \parallel \Phi \parallel^2.
\end{equation*}
The map $f$ is non-negative and proper \cite{Bis}, as follows from the properness of the moment map associated to the circle action on $\mathcal{M}_{par}(\text{Sp}\left( 2n,\mathbb{R} \right))$. Thus, a subspace $\mathcal{N}$ of $\mathcal{M}_{par}(\text{Sp}\left( 2n,\mathbb{R} \right))$ is connected if the subspace of local minima of $f$ in $\mathcal{N}$ is connected. To find the local minima of $f$, we must first determine the critical points of $f$. The following theorem indicates that critical points of $f$ are exactly the fixed points under the circle action $S^1$ on $\mathcal{M}_{par}(\text{Sp}\left( 2n,\mathbb{R} \right))$.

\begin{prop}[Proposition 3.3 in \cite{GaGoMu}]\label{209}
The proper map $f$ is a perfect Bott-Morse function. The critical points of $f$ are exactly the fixed points of the circle action. Moreover, the eigenvalue $l$ subspace for the Hessian of $f$ is the same as the weight $-l$ subspace for the infinitesimal circle action on the tangent space. In particular, the Morse index of $f$ at a critical point equals the real dimension of the positive weight space of the circle action on the tangent space.
\end{prop}

\subsection{The deformation complex}
In \cite{KySZ}, \S 3, we studied the deformation theory for parabolic $G$-Higgs bundles, when $G$ is a semisimple reductive Lie group. For $H\subset G$ a maximal compact subgroup and $\mathfrak{g}=\mathfrak{h}\oplus \mathfrak{m}$ a Cartan decomposition of the Lie algebra, the deformation complex of a parabolic $G$-Higgs bundle $\left( E,\Phi  \right)$ over $\left( X,D \right)$ is defined as the complex of sheaves
\[{{C}^{\bullet }}\left( E,\Phi  \right):E\left( {{\mathfrak{h}}^{\mathbb{C}}} \right)\to E\left( {{\mathfrak{m}}^{\mathbb{C}}} \right)\otimes K\left( D \right).\]
Moreover, the space of infinitesimal deformations of a $G$-Higgs bundle $\left( E,\Phi \right)$ is naturally isomorphic to the hypercohomology group ${{\mathbb{H}}^{1}}\left( {{C}^{\bullet }}\left( E,\Phi  \right) \right)$. Now, an element $(E, \Phi)$ is a fixed point if and only if there is an infinitesimal gauge transformation $\psi \in H^{0}(X, E(\mathfrak{g}^{\mathbb{C}}) \otimes K(D))$ such that
\begin{align*}
 d_E \psi =0, \quad [\psi, \Phi]=i \Phi.
\end{align*}
Then, $(E,\Phi)$ can be decomposed as the direct sum of eigenspaces of the gauge transformation $\psi$. We have
\begin{align*}
E(\mathfrak{g}^{\mathbb{C}})=\oplus U_m,
\end{align*}
where $\psi|_{U_m}=im$. By the relation $[\psi, \Phi]=i \Phi$, we have $\Phi:U_m \rightarrow U_{m+1} \otimes K(D)$, thus
$\Phi \in H^{0}(X,U_1 \otimes K(D))$. However, since $ E(\mathfrak{g}^{\mathbb{C}}) \cong E(\mathfrak{h}^{\mathbb{C}})\oplus E(\mathfrak{m}^{\mathbb{C}})$, and $\Phi \in H^0(X,E(\mathfrak{m}^{\mathbb{C}}) \otimes K(D))$, we conclude that $\text{ad}\Phi$ interchanges $E(\mathfrak{h}^{\mathbb{C}})$ and $E(\mathfrak{m}^{\mathbb{C}})$, whereas
\begin{align*}
E(\mathfrak{h}^{\mathbb{C}}) = \oplus U_{2m} ,\quad E(\mathfrak{m}^{\mathbb{C}}) = \oplus U_{2m+1}.
\end{align*}

The following observation is made by N. Hitchin \cite{Hit1} in the non-parabolic case, while the parabolic version can be found in \cite{GaGoMu}. If $\psi$ acts with weight $m$ on an element in $E(\mathfrak{g}^{\mathbb{C}})=\oplus U_{m}$, then the corresponding eigenvalue of the Hessian of $f$ is $-m$, while a weight $m$ on $E(\mathfrak{g}^{\mathbb{C}}) \otimes K(D)$ gives the eigenvalue $1-m$. Let $C^{\bullet }\left( E,\Phi  \right)_{-}$ be the restriction of the complex $C^{\bullet }\left( E,\Phi  \right)$ to the positive weight part. The Hessian of $f$ is negative definite on $\mathbb{H}^1(C^{\bullet }\left( E,\Phi  \right)_-)$. The positive part of the complex can be written as
\begin{align*}
C^{\bullet }\left( E,\Phi  \right)_- : \bigoplus_{m \geq 1} U_{2m} \xrightarrow{\text{ad}\Phi} \bigoplus_{m \geq 1} U_{2m+1} \otimes K(D).
\end{align*}

The next lemma will be very important in describing the connected components of $\mathcal{M}_{par}(\text{Sp}\left( 2n,\mathbb{R} \right))$. It is implied from \S 3 in \cite{GaGoMu}, similarly to the classical case of non-parabolic Higgs bundles \cite{Goth}.
\begin{lem}\label{216}
The poly-stable Higgs bundle $(E,\Phi)$ represents a local minimum of $f$ if and only if $\dim \mathbb{H}^1(C^{\bullet }\left( E,\Phi  \right)_-)=0$.
\end{lem}

\subsection{Local minima of $f$}
We describe polystable parabolic $\mathrm{Sp}(2n,\mathbb{R})$-Higgs bundles $(E,\Phi)$ representing local minima of the Hitchin functional $f:{{\mathsf{\mathcal{M}}}_{par}}\left( \text{Sp}\left( 2n,\mathbb{R} \right) \right)\to \mathbb{R}$ as defined earlier. The treatment is parallel to that by P. Gothen in the non-parabolic case \cite{Goth}.

\begin{prop}\label{406}
Let $(V,\beta,\gamma)$ be a poly-stable $\mathrm{Sp}(4,\mathbb{R})$-Higgs bundle.  $(V,\beta,\gamma)$ represents a local minimum of the parabolic Hitchin function if and only if one of the following holds
\begin{enumerate}
\item[(1)] If $\text{par}\deg V > 0$, then $\beta=0$.
\item[(2)] If $\text{par}\deg V=2g-2+s$, we have a decomposition $V=L_1 \bigoplus L_2$, and $\beta=\begin{pmatrix}
0 & 0 \\
0 & b
\end{pmatrix}$, $\gamma=\begin{pmatrix}
0 & c_1 \\
c_2 & 0
\end{pmatrix}$.
\end{enumerate}
\end{prop}

\begin{proof}
As we discussed, if $(E,\Phi)$ is a local minimum, then $E(\mathfrak{g^{\mathbb{C}}})=\oplus U_m$ can be written as the direct sum of stable Higgs bundles $U_m$ with respect to $\psi \in H^{0}(X,E(\mathfrak{g^{\mathbb{C}}}) \otimes K(D))$. Since we consider the Lie algebra $\mathfrak{g}=\mathfrak{sp}(4,\mathbb{R})$, and the complexification of $\mathfrak{h}$ appearing in the Cartan decomposition $\mathfrak{g}=\mathfrak{h}\oplus \mathfrak{m}$  is $\mathfrak{h}^{\mathbb{C}}=\mathfrak{u}(2)$, thus for $\deg V \geq 0$ there are only the following two possibilities
\begin{enumerate}
\item[(a)] $E= F_{-\frac{1}{2}} \oplus F_{\frac{1}{2}}$, $V=F_{-\frac{1}{2}}$, $V^{\vee}=F_{\frac{1}{2}}$
\item[(b)] $E= F_{-\frac{3}{2}} \oplus F_{-\frac{1}{2}} \oplus F_{\frac{1}{2}} \oplus F_{\frac{3}{2}}$, $V=F_{-\frac{3}{2}} \oplus F_{\frac{1}{2}}$, $V^{\vee}=F_{\frac{3}{2}} \oplus F_{-\frac{1}{2}}$.
\end{enumerate}
In Case (a), we have $V=F_{-\frac{1}{2}}$ and $V^{\vee}=F_{\frac{1}{2}}$ with Higgs field $\Phi: V \rightarrow V^{\vee} \otimes K(D)$. Thus the morphism $b=0$. This gives us case (1) of the proposition.\\
In Case (b), $\Phi \in H^{0}(X,E(\mathfrak{m}^{\mathbb{C}}))$ and there is a decomposition
\begin{align*}
E= F_{-\frac{3}{2}} \oplus F_{-\frac{1}{2}} \oplus F_{\frac{1}{2}} \oplus F_{\frac{3}{2}}
\end{align*}
with the Higgs field $\Phi=\begin{pmatrix}0 & \beta \\ \gamma & 0 \end{pmatrix}$, where
$\beta=\begin{pmatrix}
0 & 0 \\
0 & b
\end{pmatrix}, \gamma=\begin{pmatrix}
0 & c_1 \\
c_2 & 0
\end{pmatrix}$, for
\begin{align*}
b : F_{-\frac{1}{2}} \rightarrow F_{\frac{1}{2}} \otimes K(D), \quad c_1 : F_{\frac{1}{2}} \rightarrow F_{\frac{3}{2}}\otimes K(D), \quad c_2 : F_{-\frac{3}{2}} \rightarrow F_{-\frac{1}{2}}\otimes K(D).
\end{align*}
Based on the decomposition of $E$, $E(\mathfrak{g}^{\mathbb{C}}) = U_{-3} \oplus ... \oplus U_3$, where
\begin{align*}
& U_2 = Hom(F_{-\frac{1}{2}},F_{\frac{3}{2}})\oplus Hom(F_{-\frac{3}{2}},F_{\frac{1}{2}}),\\
& U_3 = Hom(F_{-\frac{3}{2}},F_{\frac{3}{2}}).
\end{align*}
By Proposition \ref{209}, we only have to focus on the positive weight spaces $U_2$ and $U_3$. We have
\begin{align*}
\dim \mathbb{H}^1(C_{G}^{\bullet }\left( E,\Phi  \right)_{-})& =2g-2+s+par \deg F_{\frac{3}{2}} -par \deg F_{-\frac{3}{2}} - (par \deg F_{\frac{1}{2}} -par \deg F_{-\frac{3}{2}})\\
&= 2g-2+s - (par \deg F_{-\frac{3}{2}}+par \deg F_{\frac{1}{2}})\\
&= 2g-2+s -  par \deg V,
\end{align*}
where $C_{G}^{\bullet }\left( E,\Phi  \right)_{-}$ is the restriction of the complex $C_{G}^{\bullet }\left( E,\Phi  \right)$ to the positive weight part. By Lemma \ref{216}, we get a minimum if and only if $par \deg V=2g-2+s$. This proves that $(E,\Phi)$ is of Case (b) if and only if $par \deg V =2g-2+s$. This finishes the proof of the proposition.
\end{proof}

The proof of the local minimum condition in the $\text{Sp}(2n,\mathbb{R})$ case for $n\ge 3$ is similar to the above discussion. We only give the result without a proof.

\begin{prop}\label{407}
Let $(V,\beta,\gamma)$ be a poly-stable $\mathrm{Sp}(2n,\mathbb{R})$-Higgs bundle for $n\ge 3$.  $(V,\beta,\gamma)$ represents a local minimum of the Hitchin functional if and only if one of the following holds
\begin{enumerate}
\item[(1)] If $\text{par}\deg V > 0$, then $\beta=0$.
\item[(2)] If $\text{par}\deg V=n(g-1+\frac{s}{2})$ and $m$ is odd, then there is a square root $L$ of the bundle $K(D)$ and a decomposition
\begin{align*}
    V=L K(D)^{-2 [\frac{n}{2}]} \oplus L K(D)^{-2 [\frac{n}{2}]+2} \oplus \dots \oplus L K(D)^{2 [\frac{n}{2}]}.
\end{align*}
 With respect to this decomposition, the morphisms $\beta$ and $\gamma$ have the following form
\[\beta =\left( \begin{matrix}
   0 & \cdots  & 1 & 0  \\
   \vdots  & \iddots & {} & \vdots   \\
   1 & {} & {} & {}  \\
   0 & \cdots  & {} & 0  \\
\end{matrix} \right),\,\,\,\gamma =\left( \begin{matrix}
   0 & \cdots  & 1  \\
   \vdots  &  & \vdots   \\
   1 & \cdots  & 0  \\
\end{matrix} \right),\]
where $\gamma$ is an anti-diagonal matrix.
\item[(3)] If $\text{par}\deg V=n(g-1+\frac{s}{2})$ and $n$ is even, there is a square root $L$ of the bundle $K(D)$ and a decomposition
\begin{align*}
V=L^{-1} K(D)^{2-n} \oplus L^{-1} K(D)^{4-n} \oplus \dots \oplus L^{-1} K(D)^{n}.
\end{align*}
With respect to this decomposition, the morphisms $\beta$ and $\gamma$ have the same form as in case (2) above.
\end{enumerate}
\end{prop}

\section{Parabolic Higgs Bundles vs. Higgs $V$-Bundles}
In this section we review the correspondence between parabolic Higgs bundles and Higgs $V$-bundles. Further details may be found in \cite{KySZ} and the references therein.
\subsection{Background and Correspondence}
Let $X$ be a $k$-dimensional manifold with $s$-many marked points $x_1,\dots,x_s$. For each marked point, there is a linear representation $\sigma_i: \Gamma_i \rightarrow \text{Aut}(\mathbb{R}^k)$ of a cyclic group ${{\Gamma }_{i}}=\left\langle {{\sigma }_{i}} \right\rangle $, $1 \leq i \leq s$, where $\Gamma_i$ acts freely on $\mathbb{R}^k \backslash \{0\}$ together with an atlas of coordinate charts
\begin{align*}
& \phi_i: U_i \rightarrow D^k / \sigma_i, & 1 \leq i \leq s;\\
& \phi_p: U_p \rightarrow D^k, & p \in X \backslash \{x_1,...,x_s\}.
\end{align*}
We assume that $\Gamma_i$ is the cyclic group $\mathbb{Z}_{m_i}$, where $m_i$ is a positive integer. An orbifold $M$ is obtained by gluing the local coordinate charts above, while $X$ is the underlying manifold of $M$. In \cite{FuSt}, M. Furuta and B. Steer consider this construction as a \emph{$V$-manifold}. The case we are interested in is when $M=[Y / \Gamma]$, where $Y$ is a closed, connected, smooth Riemann surface and $\Gamma$ is a finite group acting effectively on $Y$. In this case, we say that $M$ is a \emph{$V$-surface}. From the definition of the $V$-manifold, we see that the $V$-manifold is an orbifold, therefore we prefer to use the notation $[Y / \Gamma]$ to emphasize its $V$-manifold or orbifold structure. The notation $Y/\Gamma$ is the usual quotient. In this paper, $Y/\Gamma$ gives the underlying surface $X$ of $M$.

A \emph{holomorphic $V$-bundle} $E$ of rank $n$ over $M$ is defined locally on the charts as above with a collection of isotropy representations $\tau_i: \Gamma_i \rightarrow \text{Aut}(C^n)$ and local trivializations $\theta_i : E|_{U_i} \rightarrow D^k \times C^n / \sigma_i \times \tau_i$, for $1 \leq i \leq s$. Let $m=m_i$ for all $i$. The local trivialization $\Theta$ is $\mathbb{Z}_m$-equivariant with respect to the action
\begin{align*}
t(z;z_1,z_2,...,z_n)=(tz;t^{k_1}z_1,t^{k_2}z_2,...,t^{k_n}z_n).
\end{align*}

We will define the \emph{$V$-Higgs field} over the local chart $[U / \mathbb{Z}_m]$. The $V$-Higgs field defined on local charts can be glued naturally over the $V$-manifold $M$. We define $\Phi$ to be a Higgs field over the local chart $[U /\mathbb{Z}_m]$ as follows:
\begin{align*}
\Phi=(\phi_{ij})_{1 \leq i,j \leq n},
\end{align*}
where
\begin{align*}\tag{4}
\phi_{ij}=
\begin{cases}
z^{k_{i}-k_{j}} \hat{\phi}_{ij}(z^{m})\frac{dz}{z} & \text{ if } k_i \geq k_j\\
0 & \text{ if } k_i < k_j,
\end{cases}
\end{align*}
and $\hat{\phi}_{ij}$ are holomorphic functions on $\widetilde{E}$.

\begin{defn}\label{301}
A \emph{Higgs $V$-bundle} over a $V$-surface $M$ is a pair $(E,\Phi)$, where $E$ is a holomorphic $V$-bundle and $\Phi$ is a $V$-Higgs field.
\end{defn}

\begin{thm}[Theorem 6.8 in \cite{KySZ}]\label{303}
There is a bijective correspondence between isomorphism classes of holomorphic Higgs $V$-bundles with good trivialization $\left( E,\Theta ,\Phi  \right)$ and isomorphism classes of parabolic Higgs bundles $\left( F,\tilde{\Theta },\tilde{\Phi } \right)$.
\end{thm}

\begin{rem}\label{304}
In general, the $V$-Higgs field $\Phi=(\phi_{ij}) \in H^0(\text{End}_0 (E)\otimes K)$ is $\mathbb{Z}_m$-equivariant, where $\text{End}_0(E)$ is the traceless homomorphism of $E$ and the action of $\mathbb{Z}_m$ on $\text{End}_0 (E)\otimes K$ is conjugation. Under the conjugation action, we have
\begin{align*}
\phi_{ij}=
\begin{cases}
z^{k_{i}-k_{j}} \hat{\phi}_{ij}(z^{m})\frac{dz}{z} & \text{ if } k_i \geq k_j\\
z^{k_{i}-k_{j}} \hat{\phi}_{ij}(z^{m})\frac{dz}{z} & \text{ if } k_i < k_j.
\end{cases}
\end{align*}
Details can be found in \cite{NaSt}. In this paper, we slightly change the $V$-Higgs field and there are two reasons for doing so. The first reason is that if $k_i \leq k_j$, then $z^{k_{i}-k_{j}}$ possibly describes a meromorphic section, not holomorphic. The second reason is that if $\phi_{ij}$ is not trivial when $k_i < k_j$, then the corresponding Higgs field may not preserve the filtration.
\end{rem}

There is a natural way to define the degree of a holomorphic $V$-bundle $E$ on a $V$-surface (see \cite[\S 1]{FuSt}). A holomorphic $V$-bundle $E$ is \emph{stable} (resp. \emph{semistable}), if for every non-trivial proper subbundle $F \le E$ we have $\text{par}\mu \left( F \right)<\text{par}\mu \left( E \right)$ (resp. $\le $). Note that a $V$-surface is usually considered as an orbifold, which is also considered as a root stack in the language of algebraic geometry. With respect to the above stability condition, it has been shown that every semistable $V$-bundle admits a Jordan-H\"older filtration, which induces the definition of isomorphism classes of $V$-bundles, also known as the $S$-equivalent classes (see \cite{Simp2010} and \cite[\S 3.4]{Sun2020}). Based on the semistability condition and Jordan-H\"older filtration, there exists a moduli space for Higgs bundles on a $V$-surface, or more generally, on a projective Deligne-Mumford stack (see \cite[\S 9]{Simp2010} or \cite[Theorem 6.7]{Sun2020}). However, in this paper, we only focus on the correspondence between points in the moduli space of parabolic Higgs bundles and the moduli space of Higgs $V$-bundles.

Under the correspondence above (Theorem \ref{303}), we have
\begin{equation*}
\deg(E)= par\deg (F).
\end{equation*}
In conclusion, the latter equality for the degree provides the correspondence of the moduli spaces:
\begin{prop}[Proposition 5.9 in \cite{FuSt}]\label{305}
There is a bijective correspondence between isomorphism classes of holomorphic semistable (resp. stable) Higgs $V$-bundles with good trivialization $\left( E,\Theta ,\Phi  \right)$ and isomorphism classes of semistable (resp. stable) parabolic Higgs bundles $\left( F,\tilde{\Theta },\tilde{\Phi } \right)$.
\end{prop}

Indeed, we construct this special $V$-manifold from the data $(X,D,m=2)$, where $X$ corresponds to the underlying surface, $D$ includes the punctures and $m=2$ corresponds to the cyclic group action around the punctures. Therefore there is a one-to-one correspondence between the rank $n$ Higgs $V$-bundles over this special $V$-manifold and elements in $\mathcal{M}_{par}(\text{Sp}(2n,\mathbb{R})) = \mathcal{M}_{par}(X,D,n,2)$ by Theorem \ref{303}. We use this correspondence to discuss the topological invariants of $\text{Sp}(2n,\mathbb{R})$-parabolic Higgs bundles over $(X,D)$ in \S 6 later on.

\subsection{Line $V$-Bundles and the $V$-Picard Group}
With the same notation as above, we have the following proposition about line $V$-bundles.
\begin{prop}[Proposition 1.4 in \cite{FuSt}]\label{306}
Isomorphism classes of line $V$-bundles on $M$ with isotropy $\sigma_1^{k_1},\dots,\sigma_s^{k_s}$ at $x_1,\dots,x_s$ are in bijective correspondence with the integers.
\end{prop}

Proposition \ref{306} gives us a general description about all line $V$-bundles over a $V$-surface $M$. The degree of a line $V$-bundle $L$ is given by the Riemann-Roch formula for line $V$-bundles over $M$
\begin{align*}
\dim H^0(M,L)- \dim H^1(M,L)= 1- g +c_1(L)-\sum_{i=1}^s \frac{k_i}{m_i},
\end{align*}
where $c_1(L)$ is the degree of the line $V$-bundle $L$. M. Furuta and B. Steer in \cite{FuSt} showed that tuples of the form $(d,k_1,...,k_s)$ classify the topological isomorphism classes of line $V$-bundles over $M$, where $d$ is an integer and $0 \leq k_i < m_i$, $1 \leq i \leq s$. In other words, the tuple $(d,k_1,...,k_s)$ is a topological invariant for a line $V$-bundle $L$. It is easy to check that the tensor product of line $V$-bundles is still a line $V$-bundle. Therefore, the collection of all line $V$-bundles has a natural group structure.

\begin{defn}[$V$-Picard Group]\label{307}
Let $\text{Pic}_V(M)$ denote the set (or group) of isomorphism classes of line $V$-bundles over $M$.
\end{defn}

\begin{cor}[Corollary 1.6 in \cite{FuSt}]\label{308}
We have a natural map
\begin{gather*}
\text{Pic}_V(M) \longrightarrow \mathbb{Q} \oplus \left( \bigoplus_{i=1}^s \mathbb{Z}/ m_i \right)\\
L \mapsto (c_1(L),\overrightarrow{k}),
\end{gather*}
for $\overrightarrow{k}=(k_1,\dots,k_s)$ an injective homomorphism with image
$\{c_{1},(k_i \text{ mod } m_i)\}$, and where the first Chern class
$c_{1} \equiv \sum \frac{k_i}{m_i} (\text{ mod } \mathbb{Z})$.
\end{cor}

In conclusion, Corollary \ref{308} gives us the following exact sequence
 \begin{align*}
        0 \rightarrow \text{Pic}(X) \rightarrow \text{Pic}_V(M) \rightarrow \bigoplus_{i=1}^s \mathbb{Z}_{m_i} \rightarrow 0.
        \end{align*}

\section{The Beauville-Narasimhan-Ramanan Correspondence}
The Beauville-Narasimhan-Ramanan (BNR) correspondence (Proposition 3.6 in \cite{BNR}) provides a very useful tool for studying the spectral curves constructed via the Hitchin fibration. It gives a one-to-one correspondence between the category of twisted Higgs bundles $\left( E,\Phi  \right)$ over a smooth Riemann surface and line bundles over the spectral curve. An application of the correspondence is that the regular Hitchin fiber is holomorphically equivalent to the Prym variety of the spectral covering (see \cite{BaScha}, \cite{BNR}, \cite{Hit}).

In this section, we study the BNR correspondence for twisted Higgs $V$-bundles and twisted parabolic Higgs bundles. We first fix notation. Let $X$ be a Riemann surface, $D=\{x_1,...,x_s\}$ a fixed set of $s$-many points on $X$ and let $m$ be a fixed positive integer. With respect to this integer $m$, let $M$ be the corresponding $V$-surface of $X$ such that the local chart around $x_i$ is isomorphic to $\mathbb{C} / \mathbb{Z}_m$, where $\mathbb{Z}_m$ the cyclic group. The $V$-surface $M$ is uniquely determined by the data $(X,D,m)$. In \S 4.1, we discussed a special case of this construction corresponding to $(X,D,2)$. Let $\mathcal{L}$ be a line $V$-bundle over $M$, and $\mathbb{L}$ be the line bundle over $X$ under the correspondence of Theorem \ref{303}.

\subsection{BNR Correspodence for Parabolic Higgs Bundles}
Let $E \rightarrow X$ be a rank $n$ parabolic vector bundle over $X$. For each point $x \in D$, we fix a weight tuple $\alpha(x)=(\alpha_1(x),...,\alpha_n(x))$ such that $0 \leq \alpha_1(x) \leq \dots \leq \alpha_n(x) <1$ and every $\alpha_i(x)$ can be written as a fraction with denominator $m$, $ 1 \leq i \leq n$, $x \in D$. Let $\pi: \text{tot}(K(D)) \rightarrow X$ be the natural projection, where $\text{tot}(K(D))$ is the total space of the line bundle. Denote by $\lambda$ the tautological section of $\pi^*K(D)$. Given $\eta=(\eta_i)$ a set of sections, where $\eta_i$ is a section of $K(D)^i$ for $1 \leq i \leq n$, we define $X_\eta$ as the zero set of the function
\begin{align*}
\lambda^n + \eta_1 \lambda^{n-1}+ \eta_2 \lambda^{n-2} + \dots + \eta_n.
\end{align*}
We still denote by $\pi: X_\eta \rightarrow X$ the projection of the degree $n$ spectral covering. We also assume that the intersection of the set of branch points $B$ and the given divisor $D$ is empty. The importance of this assumption will be demonstrated in Remark \ref{3141} later on.

\begin{defn}\label{312}
A parabolic line bundle $(L,\widetilde{\alpha})$ over $X_\eta$ is \emph{compatible} with the given parabolic structure $\alpha$, if
\begin{align*}
\{\widetilde{\alpha}(\widetilde{x}), \widetilde{x} \in \pi^{-1}(x)\}=\{\alpha_i(x), 1 \leq i \leq n\}.
\end{align*}

Fix an order on the pre-image $\pi^{-1}(x)=\{\widetilde{x}_1,...,\widetilde{x}_n\}$ of each $x \in D$. A parabolic line bundle is \emph{strictly compatible} with the given parabolic structure $\alpha$ and the given order of the pre-image set, if
\begin{align*}
\widetilde{\alpha}(\widetilde{x}_i)=\alpha_i(x).
\end{align*}
\end{defn}

\begin{thm}\label{313}
Let $X$ be a closed Riemann surface, $D=\{x_1,...,x_s\}$ a fixed set of $s$-many points on $X$ and let $m$ be a fixed positive integer. Denote by   $K=\Omega _{X}^{1}$ the canonical line bundle over $X$  and consider $K\left( D \right):=K\otimes {{\mathsf{\mathcal{O}}}_{X}}\left( D \right)$. Fix a parabolic structure $\alpha$ for a rank $n$ parabolic bundle over $X$ and a tuple of sections $\eta=(\eta_i)$, where $\eta_i$ is a section of $K(D)^i$ for $1 \leq i \leq n$. Assume that the surface $X_\eta$ is non-singular and the intersection of the branch points $B$ and the given divisor $D$ is empty. For a fixed order of the pre-image $\pi^{-1}(x)=\{\widetilde{x}_1,...,\widetilde{x}_n\}$ of each $x \in D$, there is a bijective correspondence between isomorphism classes of strictly compatible parabolic line bundles $(L,\widetilde{\alpha})$ on $X_\eta$ and isomorphism classes of pairs $(E, \Phi)$, where $E$ is a parabolic bundle of rank $n$ with parabolic structure $\alpha$ and $\Phi:E \rightarrow E \otimes K(D)$ a parabolic Higgs field with characteristic coefficients $\eta_i$.
\end{thm}

\begin{proof}
Let $E$ be a rank $n$ vector bundle over $X$ and $L$ the corresponding line bundle over $X_\eta$ under the classical BNR correspondence (Proposition 3.6 in \cite{BNR}). Given $x \in D$, the parabolic structure over $x$ involves $n$-many rational numbers $\alpha_i(x)$, $1 \leq i \leq n$ corresponding to the weights. On the other hand, the pre-image of $x$ contains $n$ points $\{\widetilde{x}_1,...,\widetilde{x}_n\}$. Since $L$ is a line bundle, we can equip $L|_{\widetilde{x}_i}$ with a parabolic structure given by
\begin{align*}
0 \subseteq L|_{\widetilde{x}_i}, \quad 0 \leq \alpha_i(x).
\end{align*}
This construction provides a parabolic structure $\widetilde{\alpha}$ on $L$ such that $\widetilde{\alpha}(\widetilde{x}_i)=\alpha_i(x)$ and it is not hard to check that it is also a one-to-one correspondence for the parabolic structures.
\end{proof}

\begin{rem}\label{314}
Note that if we only assume that $L$ is compatible with the parabolic structure $\alpha$, but we do not fix an order for the pre-image of $x \in D$, then the correspondence may not be bijective. For instance, assume that $D$ only contains one point $x$ and consider the special cases:
\begin{itemize}
\item[(a)] if $\alpha_i(x)$ are distinct rational numbers, then the correspondence is $n!$-to-one, which is the number of all possible orderings of the $n$ points in $\pi^{-1}(x)$;
\item[(b)] if $\alpha_i$ are all the same, then the correspondence is still one-to-one.
\end{itemize}
In the general case that we have $l$ distinct weights $\alpha_i(x), 1 \leq i \leq l$, and the weight $\alpha_i(x)$ appears $k_i$ times, $1 \leq i \leq l$, the correspondence is a ${n \choose k_1,\dots,k_l}$-to-one correspondence.
\end{rem}

\begin{rem}\label{3141}
In Theorem \ref{313}, we assume that the intersection of $B$ and $D$ is empty. With respect to this assumption, each point $x \in D$ has $n$ distinct pre-images $\{\widetilde{x}_1,...,\widetilde{x}_n\}$ in $X_\eta$. Note that we have $n$ fibers $L|_{\widetilde{x}_i}$ and $n$ numbers $\alpha_i(x)$. Therefore, we can give a weight to each fiber in this case. 

If the intersection of $B$ and $D$ is not empty, without loss of generality, let $\{\widetilde{x}_1,...,\widetilde{x}_m\} \in X_\eta$ be the pre-images (as a set, not as a scheme) of $x$, where $m <n$. Since $L$ is a line bundle on $X_\eta$, we have $m$ fibers $L|_{\widetilde{x}_i}$ of dimension one. In this case, we have to equip $m$ fibers (with trivial filtration) with $n$ numbers (weights), which is impossible in the approach of Theorem \ref{313}.
\end{rem}

The above correspondence can be easily extended to the case of $\mathbb{L}$-twisted parabolic Higgs bundles.
\begin{defn}
A pair $(E,\Phi)$ is an \emph{$\mathbb{L}$-twisted parabolic Higgs bundle} over $(X,D)$, if $E$ is a parabolic Higgs bundle and $\Phi \in H^0 (\text{End}(E)\otimes \mathbb{L})$ a holomorphic section preserving the filtration.
\end{defn}

\begin{cor}
Fix a parabolic structure $\alpha$ for a rank $n$ parabolic bundle over $X$ and a tuple of sections $\eta=(\eta_i)$, where $\eta_i$ is a section of $\mathbb{L}^i$ for $1 \leq i \leq n$. Assume that the surface $X_\eta$ is non-singular and the intersection of the branch points $B$ and the given divisor $D$ is empty. For a fixed order for the pre-image $\pi^{-1}(x)=\{\widetilde{x}_1,...,\widetilde{x}_n\}$ of each $x \in D$, then there is a bijective correspondence between isomorphism classes of strictly compatible $\mathbb{L}$-twisted parabolic line bundles $(L,\widetilde{\alpha})$ on $X_\eta$ and isomorphism classes of pairs $(E, \Phi)$, where $E$ is a parabolic bundle of rank $n$ with parabolic structure $\alpha$ and $\Phi:E \rightarrow E \otimes \mathbb{L}$ is an $\mathbb{L}$-twisted parabolic Higgs field with characteristic coefficients $\eta_i$.
\end{cor}

\subsection{BNR Correspondence for $\mathcal{L}$-Twisted Higgs $V$-Bundles}
To construct a version of the correspondence for $\mathcal{L}$-twisted Higgs $V$-bundles, we first generalize the correspondence between Higgs $V$-bundles and parabolic Higgs bundles described in \S4 to the twisted case. Then, the BNR correspondence for the $\mathbb{L}$-twisted parabolic Higgs bundles directly provides a version for $\mathcal{L}$-twisted Higgs $V$-bundles.

In order to extend Theorem \ref{303} to the twisted case, we focus on the local chart $\mathbb{C} / \mathbb{Z}_m$. Consider $\mathcal{L}$ a line $V$-bundle over $M$ and $\mathbb{L}$ the corresponding parabolic line bundle over $X$. The local trivialization of $\mathcal{L}$ is $\mathbb{Z}_m$-equivariant with respect to the action
\begin{align*}
t(z;z_1)=(tz;t^{k}z_1), \quad 0 \leq k < \alpha.
\end{align*}
Let $s$ be a holomorphic section of $\mathcal{L}$. Note that $s$ is $\mathbb{Z}_m$-equivariant. In other words, it means that
\begin{align*}
(t \cdot s)(tz)=s(z).
\end{align*}
With respect to the above trivialization, the section $s$ can be written as
\begin{align*}
s(z)=z^{k}\hat{s}(z^m),
\end{align*}
where $\hat{s}$ is a holomorphic function of $\mathbb{L}$.

Let $(E,\Phi)$ be an $\mathcal{L}$-twisted Higgs $V$-bundle over $M$. The $\mathcal{L}$-twisted $V$-Higgs field $\Phi=(\phi_{ij})$ can be locally written as
\begin{align*}
\phi_{ij}=
\begin{cases}
z^{k_{i}-k_{j}} \hat{\phi}_{ij}(z^{m})(z^{k}\hat{s}(z^m)) & \text{ if } k_i \geq k_j\\
0 & \text{ if } k_i < k_j,
\end{cases}
\end{align*}
where the local trivialization is the same as the one considered in \S 4. Let $F$ be the corresponding parabolic bundle over $(X,D)$ under the correspondence of Theorem \ref{303}. The corresponding section for an $\mathbb{L}$-twisted parabolic Higgs bundle is $z^{k} \hat{\phi}_{ij}(z)\hat{s}(z^m)$. By gluing the local charts, the $\mathcal{L}$-twisted $V$-Higgs field $\Phi$ gives us an $\mathbb{L}$-twisted parabolic Higgs field $\widetilde{\Phi}$. This implies the correspondence in the twisted case. We therefore obtain the following:

\begin{prop}\label{3030}
There is a bijective correspondence between isomorphism classes of holomorphic semistable (resp. stable) $\mathcal{L}$-twisted Higgs $V$-bundles with good trivialization $\left( E,\Theta ,\Phi  \right)$ and isomorphism classes of semistable (resp. stable) $\mathbb{L}$-twisted parabolic Higgs bundles $\left( F,\tilde{\Theta },\tilde{\Phi } \right)$.
\end{prop}

Let $\eta=(\eta_i)$ be a set of sections of $\mathcal{L}^i$ for $1 \leq i \leq n$. We construct the spectral covering $\pi:M_\eta \rightarrow M$ of a $V$-surface as follows. For $X$ the underlying surface of $M$ and $D=\{x_1,...,x_s\}$, recall that the atlas of local charts of the $V$-surface $M$ over $(X,D)$ is given (see \S 4) by
\begin{align*}
& \phi_i: U_i \rightarrow D^k / \sigma_i, & 1 \leq i \leq s;\\
& \phi_p: U_p \rightarrow D^k, & p \in M \backslash \{x_1,...,x_s\}.
\end{align*}
The sections $\eta=(\eta_i)$ can be considered as a $\Gamma$-equivariant section of $X$. Thus the spectral covering $\pi: X_\eta \rightarrow X$ over $X$ is given as usual. The $V$-surface $M_\eta$ is now given as the pullback of the following diagram:
\begin{center}
\begin{tikzcd}	
	 M_\eta \arrow[d] \arrow[r] & X_\eta  \arrow[d]\\
	
	M \arrow[r] & X
\end{tikzcd}	
\end{center}

In particular, an atlas of $M_\eta$ is described by
\begin{align*}
& \phi_y: U_y \rightarrow D^k / \sigma_i, & y \in \pi_u^{-1}(D);\\
& \phi_p: U_p \rightarrow D^k, & p \in X_\eta \backslash \pi^{-1}(D).
\end{align*}
Such an atlas defines the $V$-surface $M_\eta$ over $(X_\eta, \pi^{-1}(D))$, and the underlying space of $M_\eta$ is $X_\eta$. We have a natural covering of the $V$-surface $\pi : M_{\eta} \rightarrow M$, which is the spectral covering. Under the correspondence between Higgs $V$-bundles and parabolic Higgs bundles, we have the following corollary.

\begin{cor}\label{315}
Assume that the underlying surface of $M_\eta$ is nonsingular and the intersection of the set of branch points $B$ and the divisor $D$ is empty. Then there is a bijective correspondence between isomorphism classes of strictly compatible line $V$-bundles $L$ on $M_\eta$ and isomorphism classes of pairs $(E, \Phi)$, where $E$ is a $V$-bundle of rank $n$ over $M$ and $\Phi:E \rightarrow E \otimes \mathcal{L}$ a homomorphism with characteristic coefficients $\eta_i$. For $\rho : M_\eta \rightarrow M$, it is $\rho_*(L)=E$.
\end{cor}

\begin{proof}
The proof follows from Proposition \ref{3030} (also Theorem \ref{303}) and Theorem \ref{313}. The correspondence is obtained from the following diagram:
\begin{center}
\begin{tikzcd}	
	\text{line $V$-bundles over $M_\eta$ } \arrow[r] \arrow[d, "\text{ Proposition \ref{3030}}"] & \text{ $\mathcal{L}$-twisted Higgs $V$-bundles over $M$} \arrow[d,"\text{ Proposition \ref{3030}}"]\\

	 \text{parabolic line bundles over $X_\eta$ } \arrow[r,"\text{ Theorem \ref{313}}"] & \text{ parabolic $\mathbb{L}$-twisted Higgs bundles over $X$}.
\end{tikzcd}	
\end{center}
\end{proof}

For our purposes, we are particularly interested in the following two cases of $\mathcal{L}$-twisted Higgs $V$-bundles:
\begin{enumerate}
\item[(1)] $\mathcal{L}=K(D)$ with trivial monodromy around each puncture $x \in D$;
\item[(2)] $\mathcal{L}=K$ with monodromy $\frac{1}{2}$ around each puncture $x \in D$.
\end{enumerate}
In the first case, $\frac{dz}{z}$ is a local section of $K(D)$, which is $\mathbb{Z}_2$-equivariant. It is easy to check that the corresponding local section for the chart of $X$ is $\frac{dw}{w}$, where $w=z^2$. Therefore we use the same notation $K(D)$ for the corresponding parabolic line bundle over $(X,D)$, of which the parabolic structure is trivial over each point $x \in D$. In this case, the correspondence studied in Proposition \ref{3030} and Theorem \ref{303} is exactly between $K(D)$-twisted Higgs $V$-bundles and parabolic Higgs bundles.\\
In the second case, $dz$ is a local section of $K$. Under the same calculations as in \S 4, the section $dz$ is also $\mathbb{Z}_2$-equivariant. Let $w=z^2$; then
\begin{align*}
dz=z\frac{dw}{w}.
\end{align*}
The variable can be considered as the local coordinate of the corresponding chart of $X$. More precisely, we consider the local section $\phi f(z) dz$, where $\phi$ is some holomorphic section of a $V$-bundle $\text{End}(E)$. We have
\begin{align*}
\phi f(z) dz = \phi z\frac{dw}{w}.
\end{align*}
This formula implies that the corresponding parabolic Higgs field not only preserves the filtration, but also maps the filtration $F_j$ strictly to $F_{j+1} \otimes K(D)$. Therefore the corresponding twisted parabolic Higgs bundle in this case is a strongly parabolic Higgs bundle. Details of this correspondence can be found in \cite{BiMaWo}.

The above discussion gives another interpretation of the correspondence between parabolic Higgs bundles and $\Gamma$-Higgs bundles:

\begin{prop}{\cite[\S 3]{BiMaWo}}
\begin{enumerate}
\item[(1)] Let $\mathcal{L}=K(D)$ with trivial monodromy around each puncture $x \in D$. There is a one-to-one correspondence between $\mathcal{L}$-twisted Higgs $V$-bundles and parabolic Higgs bundles.
\item[(2)] Let $\mathcal{L}=K$ with monodromy $\frac{1}{2}$ around each puncture $x \in D$. There is a one-to-one correspondence between $\mathcal{L}$-twisted Higgs $V$-bundles and strongly parabolic Higgs bundles.
\end{enumerate}
\end{prop}

\subsection{Hitchin Fibration}
Analogously to the classical BNR correspondence, Theorem \ref{313} and Corollary \ref{315} provide a way to describe the fiber of the parabolic Hitchin map ($K(D)$-twisted). Let $\text{tot}(K(D))$ be the total space of the line bundle $K(D)$. It can be written as
\[\text{tot}(K(D))=\underline{\mathrm{Spec}}\left(\mathrm{Sym}^{\bullet}((K(D))^{-1})\right).\]
We have the canonical projection $\pi:\text{tot}(K(D))\to X$. Let $\lambda\in H^0(\pi^* K(D))$ be the tautological section. The characteristic polynomial of the Higgs field
\[\det(\lambda\cdot \mathrm{id}-\pi^* \Phi)=\lambda^n+\eta_1 \lambda^{n-1}+\ldots +\eta_n\in H^0(\text{tot}(K(D)),\pi^*(K (D)^n))\]
defines the sections $\eta_i \in H^0(X,K(D)^{i})$. The parabolic Hitchin map
\begin{align*}
h:\mathcal{M}_{par}(X,D,\alpha,n)\to \bigoplus_{i=1}^n H^0(X,(K(D))^i)
\end{align*}
sends $(E,\Phi)$ to $(\eta_1,\ldots,\eta_n)$, where $(\eta_1,\ldots,\eta_n)$ is the coefficient of the characteristic polynomial of $\Phi$. If $\Phi$ is strongly parabolic, then the eigenvalues of $\Phi$ will vanish at the divisor. Therefore, $\eta_i\in H^0(X,K^iD^i\otimes \mathcal{O}_X(D)^{-1})=H^0(X,K^iD^{i-1})$, and the Hitchin fibration is the map $$h:\mathcal{M}^{s}_{par}(X,D,\alpha,n) \to H=\bigoplus_{i=1}^n H^0(X,(K(D))^i\otimes \mathcal{O}_{X}(-D))$$ sending $E$ to $\eta_i$, where $\mathcal{M}^{s}_{par}(X,D,\alpha,n)$ denotes the moduli space of strongly parabolic Higgs bundles with parabolic structure $\alpha$. More generally, we have
\begin{align*}
\mathcal{M}_{par}(X,D,n) \rightarrow \bigoplus\limits_{i=1}^n H^0(X,K(D)^i).
\end{align*}
The parabolic Hitchin map $\mathcal{M}_{par}(X,D) \rightarrow \bigoplus\limits_{i=1}^n H^0(X,K(D)^i)$ does not interact with the parabolic structure. More precisely, the parabolic Hitchin fibration of a regular point $\eta \in \bigoplus\limits^n_{i=1} H^0(X,K(D)^i)$ may contain two elements $(E,\alpha,\Phi)$ and $(E,\alpha',\Phi)$ for distinct parabolic structures $\alpha \neq \alpha'$.

Considering weights as rational numbers with denominator $m$, for a fixed positive integer $m\ge 2$, we want to describe the fiber of the parabolic Hitchin map $\mathcal{M}_{par}(X,D,n,m) \rightarrow \bigoplus\limits^n_{i=1} H^0(X,K(D)^i)$. Note that we can construct the $V$-surface $M$ over $(X,D)$ such that the action around each $x \in D$ is the cyclic group $Z_m$ of order $m$. Then any point in $\mathcal{M}_{par}(X,D,n,m)$ corresponds to a $V$-bundle over $M$. This correspondence is bijective by Theorem \ref{303}, thus providing the following:
\begin{prop}\label{316}
The moduli space $\mathcal{M}_{par}(X,D,n,m)$ is isomorphic to the moduli space of rank $n$ $V$-bundles over $M$, where $M$ is the $V$-surface uniquely determined by the data $(X,D,m)$.
\end{prop}

The Beauville-Narasimhan-Ramanan correspondence for $V$-surfaces provides now the following description of the fibers:
\begin{prop}\label{3170}
The fiber of the parabolic Hitchin map $$\mathcal{M}_{par}(X,D,n,m) \rightarrow \bigoplus\limits^n_{i=1} H^0(X,K(D)^i)$$ over a regular point $\eta \in \bigoplus\limits^n_{i=1} H^0(X,K(D)^i)$ is isomorphic to the $V$-Picard group of the spectral covering $M_\eta$ of $M$, where $M$ is uniquely determined by the data $(X,D,m)$.
\end{prop}

Now let us consider a special case: $n=2$. For the moduli space of $\text{SL}(2,\mathbb{R})$-Higgs bundles (resp. $\text{GL}(2,\mathbb{R})$-Higgs bundles), N. Hitchin \cite{Hit} proved that the fiber of the Hitchin fibration of a regular point $\eta \in H^0(X,K^2)$ with a given determinant is biholomorphic to the Prym variety of the spectral covering $X_\eta \rightarrow X$. The Prym variety $\text{Prym}(X_\eta,X)$ is defined as
\begin{align*}
\text{Prym}(X_\eta,X)=\{L \in \text{Jac}(X_\eta)| \tau^*L =L^{-1}\},
\end{align*}
where $\tau : X_\eta \rightarrow X_\eta$ is the involution of the double covering $X_\eta \rightarrow X$.

On the moduli space of parabolic $\text{SL}(2,\mathbb{R})$-Higgs bundles $\mathcal{M}_{par}(\text{SL}(2,\mathbb{R}),X,D,n)$ (resp. $\text{GL}(2,\mathbb{R})$-Higgs bundles $\mathcal{M}_{par}(\text{GL}(2,\mathbb{R}),X,D,n)$), the parabolic Hitchin fibration of a regular point $\eta$ in the Hitchin base $H^0(X,K(D)^2)$ is exactly the Prym variety
\begin{align*}
\text{Prym}(M_\eta,M)=\{L \in \text{Pic}_V(M_\eta)| \tau^*L =L^{-1}\},
\end{align*}
of the spectral covering $M_\eta \rightarrow M$, where $M$ and $M_\eta$ are the corresponding $V$-surfaces of $X$ and $X_\eta$, and $\tau:M_\eta \rightarrow M_\eta$ is the involution of $M_\eta$.

Recall that we have the following short exact sequences for the $V$-Picard groups $\text{Pic}_V(M)$ and $\text{Pic}_V(M_\eta)$
\begin{align*}
    0 \rightarrow \text{Pic}(X) \rightarrow \text{Pic}_V(M) \rightarrow \bigoplus_{i=1}^{s} \mathbb{Z}_2 \rightarrow 0,\\
    0 \rightarrow \text{Pic}(X_\eta) \rightarrow \text{Pic}_V(M_\eta) \rightarrow \bigoplus_{i=1}^{2s} \mathbb{Z}_2 \rightarrow 0.
\end{align*}
These exact sequences together with the $\Gamma$-equivariant condition give us the following short exact sequence for $\text{Prym}_V(M_\eta,M)$:
\begin{align*}
0 \rightarrow \text{Prym}(X_\eta,X) \rightarrow \text{Prym}_V(M_\eta,M) \rightarrow \bigoplus_{i=1}^{s} \mathbb{Z}_2 \rightarrow 0.
\end{align*}

The results of this section imply the following:
\begin{prop}\label{317}
The fiber of the parabolic Hitchin map
\begin{align*}
\mathcal{M}_{par}(X,D,2,2) \rightarrow H^0(X,K(D)^2)
\end{align*}
is a finite number of copies of the Prym variety $\text{Prym}(M_\eta,M)$. This number of copies only depends on the parabolic structure $\alpha$.
\end{prop}

\subsection{BNR Correspondence for $\mathfrak{L}$-Twisted Higgs Bundles over a Root Stack}
In this subsection, we discuss the BNR correspondence for $\mathfrak{L}$-twisted Higgs bundles over a root stack. Although we do not use this version of the correspondence in this paper, we include it here for completeness. In fact, the $V$-manifold has a natural structure as a root stack, thus the correspondence for $V$-bundles implies the correspondence for bundles over a root stack.

We first review the definition of a Higgs bundle over an algebraic stack $\mathfrak{X}$. A coherent sheaf $\mathcal{F}$ on $\mathfrak{X}$ is defined as follows. Let $U,V$ be two schemes, and let
\begin{align*}
f_U: U \rightarrow \mathfrak{X}, \quad f_V : V \rightarrow \mathfrak{X}
\end{align*}
be two \'{e}tale morphisms. Let $h:U \rightarrow V$ be a morphism of schemes such that the following diagram commutes
\begin{center}
\begin{tikzcd}
U \arrow[rd, "f_U"] \arrow[rr, "h"] &  & V \arrow[ld,"f_V"] \\
& \mathfrak{X} &
\end{tikzcd}
\end{center}

A \emph{coherent sheaf} $\mathcal{F}$ on $\mathfrak{X}$ consists of a collection of coherent sheaves $\mathcal{F}_U$ for each \'{e}tale morphism $f_U: U \rightarrow \mathfrak{X}$ along with isomorphisms $a_{h}:\mathcal{F}_U \rightarrow h^* \mathcal{F}_V$. The coherent sheaf $\mathcal{F}$ is \emph{locally free} if $\mathcal{F}_U$ is locally free for any \'{e}tale morphism $U \rightarrow \mathfrak{X}$. If $\mathcal{F}$ is a locally free sheaf, we say $\mathcal{F}$ is a bundle over $\mathfrak{X}$. Let $X$ be the coarse moduli space of $\mathfrak{X}$. If $X$ is a Riemann surface, denote by $\omega_{\mathfrak{X}}$ the canonical bundle over the stack $\mathfrak{X}$, which is locally defined as $\omega_{\mathfrak{X}_U} := \omega_{U/ \mathbb{C}}$.

Let $\mathcal{F}$ be a locally free sheaf on $\mathfrak{X}$. A Higgs field $\Phi$ on $\mathcal{F}$ is a homomorphism $\Phi: \mathcal{F} \rightarrow \mathcal{F} \otimes \omega_{\mathfrak{X}}$. Locally, $\Phi_U : \mathcal{F}_U \rightarrow \mathcal{F}_U \otimes \omega_U$ and the following diagram is commutative:
\begin{center}
\begin{tikzcd}
\mathcal{F}_U \arrow[d,"a_h"] \arrow[r,"\Phi_U"]  & \mathcal{F}_U \otimes \omega_U \arrow[d,"a_h^{\mathcal{F}} \otimes a_h^{\omega_{\mathfrak{X}}}"] \\
h^* \mathcal{F}_V  \arrow[r,"\Phi_V"]  & h^* \mathcal{F}_V \otimes h^* \omega_V
\end{tikzcd}
\end{center}

Now we consider the root stack. Let $\mathfrak{X}$ be an algebraic stack and let $L$ be an invertible sheaf over $\mathfrak{X}$. We fix a positive integer $r$ and take a section $s \in \Gamma(\mathfrak{X},L)$. The pair $(L,s)$ provides a morphism $\mathfrak{X} \rightarrow [\mathbb{A}^1 / \mathbb{G}_m]$. Denote by $\theta_r : [\mathbb{A}^1 / \mathbb{G}_m] \rightarrow [\mathbb{A}^1 / \mathbb{G}_m]$ the morphism induced by the $r$-th power maps on both $\mathbb{A}^1$ and $\mathbb{G}_m$. Clearly, the map $\theta_r$ sends a pair $(L,s)$ to its $r$-th tensor product $(L^r,s^r)$. The root stack $\mathfrak{X}_{(L,s,r)}$ is defined as follows
\begin{align*}
\mathfrak{X}_{(L,s,r)}:=\mathfrak{X} \times_{[\mathbb{A}_1 / \mathbb{G}_m], \theta_r}[\mathbb{A}_1 / \mathbb{G}_m].
\end{align*}
As a special case, let $\mathfrak{X}=X$ be a scheme. Then the objects of $X_{(L,s,r)}$ over a scheme $S$ are quadruples $(f,M,t,\varpi)$, where $f: S \rightarrow X$ is a morphism, $M$ is an invertible sheaf on $S$, $t$ is a section in $\Gamma(S,M)$ and $\varpi : M^r \rightarrow f^* L$  is an isomorphism such that $\varpi(t^r) = f^*s$. If $X = \mathbb{A}^1$ with coordinate $x$, then we have
\begin{align*}
X_{\mathcal{O}_X,x,r} \cong [\mathbb{A}^1 / \mu_r].
\end{align*}

Now we consider the root stack $X_{\mathcal{O}_X(D),s,r}$, and in the rest of this subsection, the notation $\mathfrak{X}$ shall denote the root stack $X_{\mathcal{O}_X(D),s,r}$. For each point $x \in D$, there is an open affine neighborhood $\text{Spec}(A_x)$ of $x$ such that $\mathfrak{X}$ is locally isomorphic to 
\begin{align*}
[\left(\text{Spec}(A_x[t_x]/(t_x^r-s_x))\right)/\mu_r],
\end{align*}
where $s_x$ is the restriction of $s$ to $\text{Spec}(A_x)$. Abusing the language of orbifolds and root stacks, $\mathfrak{X}$ is locally isomorphic to $[\mathbb{C}/\mathbb{Z}_m]$ (see the notation in \S 4.1). Therefore, the results in \S 5.2 (Proposition \ref{3030} and Corollary \ref{315}) still hold for root stacks, and we omit the proofs.

I. Biswas, S. Majumder and M. L. Wong proved an equivalence of categories of Higgs bundles on $\mathfrak{X}$ and strongly parabolic Higgs bundles on $X$, where $\mathfrak{X}=X_{\mathcal{O}_X(D),s,r}$ (Theorem 4.7 in \cite{BiMaWo}). Their approach can be easily extended to the $\mathfrak{L}$-twisted case. Let $\mathfrak{L}$ be a line bundle over $\mathfrak{X}$, and let $\mathbb{L}$ be the parabolic line bundle over $(X,D)$ corresponding to $\mathfrak{L}$. With respect to the above notation, we have the following proposition.
\begin{prop}\label{318}
There is a one-to-one correspondence between $\mathfrak{L}$-twisted Higgs bundles over $\mathfrak{X}$ and $\mathbb{L}$-twisted parabolic Higgs bundles over $(X,D)$.
\end{prop}
The idea in the proof of this proposition is similar to those of Theorem \ref{303} and Proposition \ref{3030}.

Now Let $\eta=(\eta_1,...,\eta_n) \in \bigoplus_{i=1}^n H^0(\mathfrak{X},\mathfrak{L}^i)$ be an element in the Hitchin base for the stack $\mathfrak{X}$. We can define the spectral covering $\mathfrak{X}_\eta$ as the zero locus of
\begin{align*}
\lambda^n+\eta_1 \lambda^{n-1} + \dots + \eta_n.
\end{align*}
Let $X_\eta$ be the underlying space of $\mathfrak{X}_\eta$. Denote by $B$ the set of branch points in $\mathfrak{X}$ with respect to the covering $\mathfrak{X}_\eta \rightarrow \mathfrak{X}$. We assume that the intersection of $B$ and $D$ is empty. Since $\mathfrak{X}=X_{\mathcal{O}_X(D),s,r}$ is a root stack, $\mathfrak{X}_\eta$ is also a root stack with the same argument for $M_\eta$ in \S 5.2. The version of the correspondence for a root stack is then given in the following:

\begin{cor}[BNR Correspondence for Stacks]\label{319}
There is a one-to-one correspondence between line bundles over $\mathfrak{X}_\eta$ and $\mathfrak{L}$-twisted Higgs bundles over $\mathfrak{X}$.
\end{cor}

\begin{proof}
Similar to the proof of Corollary \ref{315}, this Corollary is directly implied by Proposition \ref{3030} and Proposition \ref{318}.
\end{proof}

\subsection{Conclusion}
With respect to the discussion in \S4 and \S5, we have one-to-one correspondences among the following six categories:
\begin{center}
\begin{tikzcd}[column sep=12ex]	
	\text{$\mathcal{L}$-twisted Higgs $V$-bundles over $M$ } \arrow[r,"\text{ Corollary \ref{315}}"] \arrow[d,"\text{ Proposition \ref{3030}}"] & \text{ line $V$-bundles over $M_\eta$} \arrow[d,"\text{ Proposition \ref{3030}}"]\\

	 \text{ parabolic $\mathbb{L}$-twisted Higgs bundles over $X$} \arrow[r,"\text{ Theorem \ref{313}}"] \arrow[d,"\text{ Proposition \ref{318}}"] & \text{parabolic line bundles over $X_\eta$} \arrow[d,"\text{ Proposition \ref{318}}"] \\
	
	\text{ $\mathfrak{L}$-twisted Higgs bundles over stack $\mathfrak{X}$} \arrow[r,"\text{ Corollary \ref{319}}"] & \text{ line bundles over $\mathfrak{X}_\eta$}.
\end{tikzcd}	
\end{center}

In \S 7 later on, we will use the BNR correspondence and the $V$-Prym variety to prove the connectedness of the components of the moduli space $\mathcal{M}_{par}(\text{Sp}(2n,\mathbb{R}))$ under several considerations for the parabolic structure.

\section{Topological Invariants of the Moduli Space of Parabolic $\text{Sp}(2n,\mathbb{R})$-Higgs Bundles}

\begin{defn}
Let $\left( V,\beta ,\gamma  \right)$ be a parabolic $\text{Sp}\left( 2n,\mathbb{R} \right)$-Higgs bundle over $\left( X,D \right)$. The \emph{parabolic Toledo invariant} of $\left( V,\beta ,\gamma  \right)$ is defined as the rational number $\tau =par\deg \left( V \right)$.
\end{defn}
In \cite{KySZ}, it is shown that for a semistable parabolic $\text{Sp}\left( 2n,\mathbb{R} \right)$-Higgs bundle, the parabolic Toledo invariant satisfies the inequality
\[\left| \tau  \right|\le n\left( g-1+\frac{s}{2} \right),\]
where $s$ is the number of points in the divisor $D$. We call  \emph{maximal} the parabolic $\text{Sp}(2n,\mathbb{R})$-Higgs bundles for which $\tau =n\left( g-1+\frac{s}{2} \right)$ and denote the components containing those by
\[\mathsf{\mathcal{M}}_{par}^{\max }(\text{Sp}(2n,\mathbb{R})):=\mathsf{\mathcal{M}}_{par}^{n\left( g-1+\frac{s}{2} \right)}(\text{Sp}(2n,\mathbb{R})).\]

In this maximal case and for a fixed square root $L_0$ of $K(D)$ there is a one-to-one correspondence, called the \emph{parabolic Cayley correspondence}, which maps every maximal triple $\left( V,\beta ,\gamma  \right)$ to a triple $(W,\phi,c)$ for $W:=V \otimes L^{-1}_0$ and for $\phi,c$ defined by
\begin{align*}
& \phi=(\beta \otimes 1_{L_0}) \circ (\gamma \otimes 1_{L^{-1}_0}): W=V\otimes L^{-1}_0 \rightarrow V \otimes L^{\frac{3}{2}}_0=W\otimes K(D)^2,\\
& c=\gamma \otimes 1_{L^{-1}_0} : W=V \otimes L^{-1}_0 \rightarrow V^{\vee} \otimes K(D) \otimes L^{-1}_0=W^{\vee}.
\end{align*}
Note also that the map $\gamma$ is an isomorphism in the maximal case, implying that $c$ is also an isomorphism and $\phi:W \rightarrow W \otimes K(D)^2$ is a $K(D)^2$-twisted parabolic Higgs field for $W$. Moreover, the parabolic structure on the bundle $V$ provides the construction of the corresponding $V$-surface $M$, where $M$ is the $V$-manifold with $s$-many marked points $p_1,...,p_s$, around which the isotropy group is $\mathbb{Z}_2$, and $X$ is the underlying surface of $M$. Let us use the same notation $(V,\Phi)$ or $(V,\beta,\gamma)$ for the corresponding $V$-Higgs bundle over $M$. On the other hand, since the parabolic degree is equal to the degree as a $V$-bundle, we use the notation $par\deg(L)$ to also mean the degree of $L$ as a $V$-bundle, while $\deg(L)$ is the degree of $L$ over the underlying surface $X=|M|$. As a $V$-bundle morphism, the isomorphism $c$ induces a quadratic form on $W$. Hence, the structure group of $W$ is $\text{O}(n,\mathbb{C})$.

\subsection*{Topological Invariants of $\mathcal{M}^{max}_{par}(\text{Sp}(2n,\mathbb{R}))$}

We determine the topological invariants for maximal parabolic $\text{Sp(2n}\text{,}\mathbb{R}\text{)}$-Higgs bundles as elements in $H^1_V(M,\mathbb{Z}_2)$, the first $V$-cohomology group. We briefly review next the definition and calculations of $V$-cohomology groups (see \cite[\S 7]{KySZ} for more details).

One can consider a $V$-surface $M$ by gluing two charts $U_1$ and $U_2$, for $U_1= X \backslash \{x_1,...,x_s\}$ and $U_2 = \coprod_{i=1}^s D / \mathbb{Z}_2$. With respect to this atlas, we define $M_V$ as
\begin{align*}
M_V=V_1 \bigcup V_2, \quad V_1 = X \backslash \{x_1,...,x_s\}, \quad V_2 = \coprod_{i=1}^s D\times_{\mathbb{Z}_2} E\mathbb{Z}_2,
\end{align*}
where $D$ is a disk around the punctures $x_{i}$ and $X$ is a compact Riemann surface of genus $g$. Define the $V$-cohomology group as
\begin{align*}
H^i_V(M,\mathbb{Z}_2):=H^i(M_V,\mathbb{Z}_2).
\end{align*}
If there is no confusion, we omit $\mathbb{Z}_2$ in the cohomology group.

Now we are ready to calculate the rank of $H^1(M_V)$ and $H^2(M_V)$. By the Mayer-Vietoris sequence, we have
\begin{align*}
0 \rightarrow H^0(M_V) \rightarrow H^0(V_1) \bigoplus H^{0}(V_2) \rightarrow H^0(V_1\bigcap V_2)\\
\xrightarrow[]{j_1}  H^1(M_V) \rightarrow H^1(V_1) \bigoplus H^{1}(V_2) \rightarrow H^1(V_1\bigcap V_2)\\
\xrightarrow[]{j_2} H^2(M_V) \rightarrow H^2(V_1) \bigoplus H^{2}(V_2) \rightarrow H^2(V_1\bigcap V_2).\\
\end{align*}
Note that $V_1=X \backslash \{x_1,...,x_s\}$ and $V_1\bigcap V_2=\prod_{i=1}^s S^1$. We have
\begin{align*}
& \text{rk}(H^0(V_1\bigcap V_2))=s, \quad \text{rk}(H^0(V_1))=1, \\
& \text{rk}(H^1(V_1\bigcap V_2))=s, \quad \text{rk}(H^1(V_1))=2g+s-1, \\
& \text{rk}(H^2(V_1\bigcap V_2))=0, \quad \text{rk}(H^2(V_1))=0,
\end{align*}
where $\text{rk}(\bullet)$ denotes the rank of the cohomology group with coefficient in $\mathbb{Z}_2$.
We use the Leray spectral sequence to calculate the cohomology group of $V_2$, and we have
\begin{align*}
    \text{rk}(H^0(V_2))=s,\\
    \text{rk}(H^1(V_2))=s,\\
    \text{rk}(H^2(V_2))=s.
\end{align*}
Therefore,
\begin{align*}
& \text{rk}(H^1(M_V))=2g+s-1,\\
& \text{rk}(H^2(M_V))=s.
\end{align*}

For the $\text{Sp(4}\text{,}\mathbb{R}\text{)}$-case we see the following:
\begin{enumerate}
\item[(1)] If $\bigwedge^2 W \neq 0$, the structure group of $W$ is $\text{O}(2,\mathbb{C})$ and every pair $(u,v)$ is a topological invariant for $W$, where $0 \neq u \in H^1_V(M,\mathbb{Z}_2)$ and $v \in H^2_V(M,\mathbb{Z}_2)$. Thus, there are clearly $2^s(2^{2g+s-1}-1)$ many distinct pairs $(u,v)$. Denote by $\mathcal{M}_{par}^{u,v}$ the moduli space corresponding to $(u,v)$.
\item[(2)] If $\bigwedge^2 W=0$, then the structure group can be reduced to $\text{SO}(2, \mathbb{C}) \subset \text{O}(2,\mathbb{C})$. From the identification $\text{SO}(2,\mathbb{C}) \cong \mathbb{C}^*$, $W$ can be decomposed as a direct sum $W=L \oplus L^{\vee}$, where $L$ is a line $V$-bundle and $L^{\vee}$ is the dual of $L$. Now, stability for the map $\phi:W \rightarrow W \otimes K(D)^2$, provides the existence of a non-trivial holomorphic map $L \rightarrow L^{\vee}\otimes K(D)^2$, therefore it holds necessarily that $par\text{deg}(L) \leq 2g-2+s$.
    \begin{enumerate}
    \item[a.] If $par\deg(L)\neq 2g-2+s$, then every possible value of the degree gives at least one topological invariant, thus providing at least $2g-2+s$ different non-negative values. Now we want to describe all possible topological invariants for line $V$-bundles with degree smaller than $2g-2+s$. Recall from \S 4.2 that the $V$-Picard group $\text{Pic}_V(M)$ classifies all line $V$-bundles over $M$. The $V$-Picard group has the following short exact sequence
        \begin{align*}
        0 \rightarrow \text{Pic}(X) \rightarrow \text{Pic}_V(M) \rightarrow \bigoplus_{x \in D} \mathbb{Z}_2 \rightarrow 0.
        \end{align*}
By definition, the parabolic degree is the sum of the classical degree and the weights
        \begin{align*}
        par \deg =\deg + \text{weights}.
        \end{align*}
In fact, we can understand this definition from the above exact sequence about $\text{Pic}_V(M)$. The classical $\deg$ comes from $\text{Pic}(X)$ and the weights come from $\bigoplus\limits_{x\in D} \mathbb{Z}_2$. From the exact sequence, it is also clear that $\text{Pic}_V(X)$ is not necessarily connected. Then every pair $(d,w)$, where $d \in \text{Pic}(X)$ the classical degree and $w \in \bigoplus\limits_{x\in D} \mathbb{Z}_2$ the weight, describes a topological invariant for the line $V$-bundles. Although different pairs $(d,w)$ may have the same parabolic degree, they provide different topological invariants for line $V$-bundles. Denote by $\mathcal{M}_{par}^{0,(d,w)}$ the moduli space in this case.

    \item[b.] If $par\deg(L)=2g-2+s$, we have $L^2 \cong K(D)^2$. This describes parabolic $\text{Sp}(4,\mathbb{R})$-Higgs bundles $\left( E=V\oplus {{V}^{\vee }},\Phi  \right)$ with $V=N\oplus {{N}^{\vee }}K\left( D \right)$, for a line bundle  $N=K{{\left( D \right)}^{\frac{3}{2}}}$. Thus, square roots of $K\left( D \right)$ parameterize components containing such Higgs bundles, and this contributes to $2^{2g+s-1}$ topological invariants. Denote by $\mathcal{M}_{par}^{0,2g-2+s,L}$ the moduli space in this case.
    \end{enumerate}
\end{enumerate}

The investigation of the $\text{Sp}(2n,\mathbb{R})$ case, $n \geq 3$, is analogous to the one for $\text{Sp}(4,\mathbb{R})$. We define the Cayley partner $(W,c,\phi)$ similarly and describe the topological invariants for all possible Cayley partners. By Proposition \ref{407}, we can discuss this problem in two cases: $\beta=0$ and $\beta \neq 0$.

\begin{enumerate}
\item If $\beta=0$, it means that the corresponding morphism $\phi$ in the Cayley partner is also trivial. Therefore the Cayley partner reduces to $(W,c)$, where $c$ is an isomorphism of the parabolic $\text{GL}(n,\mathbb{R})$-bundle $W$, or equivalently the $V$-bundle $W$. Thus the $V$-cohomology groups $H^1_V(M,\mathbb{Z}_2)$, $H^2_V(M,\mathbb{Z}_2)$ are topological invariants for such bundle $W$. In this case, we have $2^s \cdot 2^{2g+s-1}$ many topological invariants. Denote by $\mathcal{M}_{par}^{u,v}(\text{Sp}(2n,\mathbb{R}))$ the component of $V$-Higgs bundle $(E,\Phi)$ such that the topological invariants of its Cayley partner $(W,\phi,c)$ are $u \in H^1_V(M,\mathbb{Z}_2)$ and $v \in H^2_V(M,\mathbb{Z}_2)$.
\item If $\beta \neq 0$, in either Case (2) or Case (3) of Proposition \ref{407} one needs to fix a square root $L$ of the bundle $K(D)$. Note that we are taking the square root of a $V$-bundle. Denote by $\mathcal{M}_{par}^{0,n(g-1+\frac{s}{2}),L}(\text{Sp}(2n,\mathbb{R}))$ the moduli space in this case and there are $2^{2g+s-1}$ many such choices for the square root $L$.
\end{enumerate}

\section{Connected Components of $\mathcal{M}^{max}_{par}(\text{Sp}(2n, \mathbb{R}))$}
Based on the previous two sections, we will now prove that the moduli spaces $\mathcal{M}_{par}^{u,v}$, $\mathcal{M}_{par}^{0,(d,w)}$ and $\mathcal{M}_{par}^{0,2g-2+s,L}$ for each possible value of the corresponding topological invariants are connected. As explained in \S 3, it is enough to prove that the subspace of local minima $\mathcal{N}$ for each of those components is connected. The proofs below are for the case $G=\text{Sp}\left( 4,\mathbb{R} \right)$; for the case  $G=\text{Sp}\left( 2n,\mathbb{R} \right)$ when $n\ge 3$ the proofs are no different, thus we state directly the analogous results.

\begin{prop}\label{501}
$\mathcal{M}_{par}^{u,v}$ is connected, where $u\neq 0$.
\end{prop}

\begin{proof}
As seen in Proposition \ref{406}, $u={{w}_{1}}\left( W \right)\ne 0$ only for triples $\left( V,\beta ,\gamma  \right)$ of the type (1), that is, for those when $\beta =0$. This implies that the Higgs field $\phi $ in the Cayley partner vanishes, hence we consider the moduli space for the pairs $\left( W,c \right)$.\\
Notice that the first $V$-cohomology group is isomorphic to $\text{Hom}(\pi_V^1(M),\mathbb{Z}_2)$. Moreover,
\begin{align*}
\text{Hom}(\pi_V^1(M),\mathbb{Z}_2) \cong \text{Hom}(\pi_1(X / D),\mathbb{Z}_2).
\end{align*}
Thus, given an element $u \in H^1_V(M,\mathbb{Z}_2)$, we can construct a degree two \'{e}tale covering over the underlying surface with punctures
\begin{align*}
\pi_u: X_u/ \pi_u^{-1}(D) \rightarrow X/ D,
\end{align*}
where $X_u$ is a compact surface. The existence of $X_u$ is provided by Theorem 1.1 in \cite{FuSt}. Clearly,
\begin{align*}
\pi_u^* W |_{X/D}=L |_{X/D} \bigoplus L^{-1} |_{X/D},
\end{align*}
where $L$ is a line bundle over $X_u$, with $\pi^* c$ an isomorphism of $\pi_u^* W$. A $V$-surface $M_{X_u}$ with underlying surface $X_u$ can be constructed with local charts given by
\begin{align*}
& \phi_y: U_y \rightarrow D^k / \mathbb{Z}_2, & y \in \pi_u^{-1}(D);\\
& \phi_p: U_p \rightarrow D^k, & p \in M \backslash \pi_u^{-1}(D).
\end{align*}
There is also a natural covering of $V$-surface $\pi_u : M_{X_u} \rightarrow M$, thus there are three covering morphisms summarized below
\begin{align*}
\text{ \'{e}tale: }& X_u/ \pi_u^{-1}(D) \rightarrow X/ D,\\
\text{ ramified: }& X_u \rightarrow X,\\
\text{ $V$-surface: }& M_{X_u} \rightarrow M.
\end{align*}
We shall use the same notation $\pi_u$ for all these coverings.
\begin{center}
\begin{tikzcd}	
	 X_u/ \pi_u^{-1}(D) \arrow[d, "\pi_u" ] \arrow[r,hook] & X_u  \arrow[d, "\pi_u"]
	& M_{X_u} \arrow[l,Rightarrow] \arrow[d, "\pi_u" ] \\
	
	X/ D \arrow[r,hook] & X
	& M \arrow[l,Rightarrow]
\end{tikzcd}	
\end{center}
The direct sum
\begin{align*}
\pi_u^* W |_{X/D}=L |_{X/D} \oplus L^{-1} |_{X/D}
\end{align*}
can be extended to the $V$-surface $M_{X_u}$. Let $\tau : M_{X_u} \rightarrow M_{X_u}$ be the involution interchanging the coverings and clearly $\tau^* L = L^{-1}$. From the discussion following Proposition \ref{3170}, $L$ is a well-defined element in the Prym variety $\text{Prym}(M_{X_u},M)$ of $V$-surfaces.\\
We now turn to these line $V$-bundles $L$ in the Prym variety $\text{Prym}(M_{X_u},M)$. In our case, the monodromy group is $\mathbb{Z}_2$, hence there are only two possible monodromy actions around $p \in \pi_u^{-1}(D)$, namely $0$ and $\frac{1}{2}$. If $p_1,p_2 \in \pi_u^{-1}(D)$ and $\pi_u(p_1)=\pi_u(p_2)$, then the monodromy action around these two points should be the same described by the condition $\tau^{*}L=L^{-1}$. The short exact sequence
\begin{align*}
    0 \rightarrow \text{Pic}(X_u) \rightarrow \text{Pic}_V(M_{X_u}) \rightarrow \bigoplus_{i=1}^{s} \mathbb{Z}_2 \rightarrow 0
\end{align*}
implies the following sequence for the Prym variety
\begin{align*}
    0 \rightarrow \text{Prym}(X_u,X) \rightarrow \text{Prym}(M_{X_u},M) \rightarrow \bigoplus_{i=1}^{s} \mathbb{Z}_2 \rightarrow 0.
\end{align*}
Since the covering $X_u \rightarrow X$ is a ramified covering, $\text{Prym}(X_u,X)$ is connected. Therefore, each element in $\bigoplus_{i=1}^{s} \mathbb{Z}_2$ provides a connected component of the Prym variety $\text{Prym}(M_{X_u},M)$.\\
The above argument implies that given a weight $w$ and an element in the first $V$-cohomology group $u \in \text{Hom}(\pi_V^1(M),\mathbb{Z}_2)$, the space of all such pairs $(W,c)$ is connected. Now we only have to show that we can use the second $V$-cohomology $H^2_V(M,\mathbb{Z}_2)$ to describe all possible weights in $\bigoplus\limits_{i=1}^{s} \mathbb{Z}_2$.\\
For the $V$-manifold $M_{V}$ of $M$, the exact sequence
\begin{align*}
0 \rightarrow \mathbb{Z} \rightarrow O_{M_V} \rightarrow O^*_{M_V} \rightarrow 0
\end{align*}
provides that
\begin{align*}
H^1_V(M, O_{M_V}) \rightarrow H^1_V(M, O^*_{M_V}) \xrightarrow{\varpi}  H^2_V(M, \mathbb{Z}).
\end{align*}
The first cohomology group $H^1_V(M, O^*_{M_V})$ is exactly the $V$-Picard group $\text{Pic}_V(M)$. Therefore, there is a morphism
\begin{align*}
\varpi:\text{Pic}_V(M) \rightarrow  H^2_V(M, \mathbb{Z}).
\end{align*}
By taking the $\mathbb{Z}_2$ coefficient, it is easy to see that the morphism $\varpi$ induces the isomorphism
\begin{align*}
\bigoplus_{i=1}^{s} \mathbb{Z}_2 \cong H^2_V(M, \mathbb{Z}).
\end{align*}
\end{proof}

\begin{cor}\label{502}
The number of all the connected components of the form $\mathcal{M}_{par}^{u,v}$ is $2^s(2^{2g+s-1}-1)$.
\end{cor}

\begin{proof}
Follows from the fact that $\text{rk}(H^1(M_V))=2g+s-1$ and $\text{rk}(H^2(M_V))=s$.
\end{proof}

\begin{prop}\label{503}
$\mathcal{M}_{par}^{0,(d,w)}$ is connected.
\end{prop}

\begin{proof}
Denote by $\mathcal{N}_{par}^{0,(d,w)}$ the subspace of local minima of $\mathcal{M}_{par}^{0,(d,w)}$. We show that $\mathcal{N}_{par}^{0,(d,w)}$ is connected. From the discussion in \S 4.2 and \S 5.3, we know the $V$-Picard group $\text{Pic}_V(M)$ is not necessarily connected. If we consider the subgroup $\text{Pic}_{w}(M)$, which consists of all line $V$-bundles with a given fixed weight $w$, this subgroup $\text{Pic}_{w}(M)$ is isomorphic to the classical Picard group $\text{Pic}(X)$ and so $\text{Pic}_{w}(M)$ is connected. Denote by $\text{Pic}_w^d(X)$ the subvariety of $\text{Pic}_{w}(X)$ with fixed degree $d$. It is clear that any line $V$-bundle in $\text{Pic}_w^d(X)$ has the same parabolic degree; denote by $l$ the corresponding parabolic degree. We shall prove the statement in two cases, namely when $d = 0$ and when $d>0$.\\
When $d=0$, we claim that the parabolic Higgs field $\Phi$ is zero. In other words, the parabolic Higgs field of any point in $\mathcal{N}_{par}^{0,(0,w)}$ is zero. By Proposition \ref{406}, the moduli space $\mathcal{N}_{par}^{0,(0,w)}$ of local minima is isomorphic to the moduli space of pairs $(W,c)$, where $W= L \oplus L^{\vee}$, $L$ is a degree $0$ line bundle, and $c = \begin{pmatrix}
0 & 1 \\
1 & 0
\end{pmatrix}$ with respect to the above decomposition. There is a surjective continuous map $Pic_w^{0}(X) \rightarrow \mathcal{N}_{par}^{0,(0,w)}$ given by taking $L$ to the pairs $(W,c)$. Hence, the moduli space $\mathcal{N}_{par}^{0,(0,w)}$ of local minima is connected. Now we claim that the parabolic Higgs field $\Phi$ is zero. If $\Phi$ is not zero, then $L^{\vee}$ is a $\Phi$-invariant sub-line bundle of $W$. This means that $W$ is semistable, but not stable, and thus finishes the proof for the first case $d=0$.\\
In the second case $d >0$, the parabolic Higgs field $\Phi$ cannot be zero, otherwise the subbundle $L$ would violate the stability condition. Therefore, critical points should be as described in case (2) of Proposition \ref{406}. The parabolic Higgs field $\Phi$ of a critical point can be written in the form
$\Phi=\begin{pmatrix}
0 & 0 \\
\gamma & 0
\end{pmatrix},$
where $\gamma \in H^{0}\left(X, ( L^{\vee}K(D) )^2 \right)$. Thus, the subvariety $\mathcal{N}_{par}^{0,(d,w)}$ fits into the following pullback diagram
\begin{center}
\begin{tikzcd}	
	& \text{Pic}_w^d(X) \arrow[d,"\text{$\cong$}"] \\
	
	\mathcal{N}_{par}^{0,(d,w)} \arrow[r] \arrow[d, "\pi"]
	& \text{Pic}^{d}(X) \arrow[d, "\text{ $L \rightarrow (L^{\vee}K(D))^2 $ }" ] \\
	
	\text{S}^{4g-4+2s-2l}(X) \arrow[r, "\text{ $D \rightarrow [D]$ }" ]
	& \text{Pic}^{4g-4+2s-2l}(X)
\end{tikzcd}	
\end{center}
where $\pi(W,C,\Phi)=(\Phi)$. This shows that $\mathcal{N}_{par}^{0,(d,w)}$ is connected also when $d >0$.
\end{proof}

\begin{cor}\label{504}
The number of all the connected components of the form $\mathcal{M}_{par}^{0,(d,w)}$ is $(2g-2+s)2^{s}$.
\end{cor}

\begin{proof}
We calculate the number of all possible pairs $(d,w)$, where $d \in \text{Pic}(X)$ is the classical degree of the parabolic line bundle and $w$ is an element in $\bigoplus\limits_{i=1}^s \mathbb{Z}_2$; it is a combinatorial problem to determine the number of all pairs $(d,w)$. Let $k(w)$ be the corresponding weight of $w$ which is a rational number with denominator $2$. In general,
\begin{align*}
k(w)=\frac{\text{ the number of $1$'s in $w$ }}{2}.
\end{align*}
Given an integer $0 \leq k \leq s$, there are ${s \choose k}$ many choices for the weight $w$ such that $2k(w)=k$.\\
Now we fix a weight $w \in \bigoplus\limits_{i=1}^s \mathbb{Z}_2$. We calculate the number of all possible degrees $d$ depending on the parity of $2k(w)$. By the definition of parabolic degree, we have
\begin{align*}
0 \leq d+k(w)<2g-2+s.
\end{align*}
Therefore, the range for $d$ is
\begin{align*}
-k(w) \leq d < 2g-2+s-k(w).
\end{align*}
Fixing $k(w)$, the number of all possible values for $d$ is $2g-2+s$. Based on the above discussion, the number of all possible pairs $(d,w)$ is
\begin{align*}
\sum_{2k(w)=0}^{s} {s \choose 2k(w)}(2g-2+s)= (2g-2+s)2^{s}.
\end{align*}
\end{proof}

\begin{prop}\label{505}
$\mathcal{M}_{par}^{0,2g-2+s,L}$ is connected.
\end{prop}

\begin{proof}
In this case, the bundle $W$ is completely determined by $L$. Thus the moduli space $\mathcal{M}_{par}^{0,2g-2+s,L}$ is isomorphic to the moduli space of parabolic Higgs fields $\Phi \in H^0(\text{End}(W) \otimes K(D)^2)$ such that $\Phi$ commutes with the matrix $c=\begin{pmatrix}
0 & 1 \\
1 & 0
\end{pmatrix}$. Therefore, the Higgs field $\Phi$ (more precisely the map $\gamma$) is of the following form
\begin{center}
$\Phi=\begin{pmatrix}
\phi_{11} & \phi_{12} \\
\phi_{12} & \phi_{22}
\end{pmatrix}$.
\end{center}
Hence, the moduli space $\mathcal{M}_{par}^{0,2g-2+s,L}$ is isomorphic to the connected space
\begin{align*}
\underbrace{H^0(X,(K(D))^2)}_{\phi_{11}} \bigoplus \underbrace{H^0(X,(K(D))^4)}_{\phi_{12}} \bigoplus \underbrace{H^0(X,(K(D))^2)}_{\phi_{22}}.
\end{align*}
\end{proof}

\begin{cor}\label{506}
The number of all the connected components of the form $\mathcal{M}_{par}^{0,2g-2+s,L}$ is $2^{2g+s-1}$.
\end{cor}

Combining the results of this section, we deduce an exact component count for the moduli space of maximal parabolic $\text{Sp}\left( 4,\mathbb{R} \right)$-Higgs bundles:

\begin{thm}\label{507}
The moduli space ${{\mathsf{\mathcal{M}}}_{par}^{\text{max}}}\left( \text{Sp}\left( 4,\mathbb{R} \right) \right)$ of parabolic maximal $\text{Sp}\left( 4,\mathbb{R} \right)$-Higgs bundles with all weights rational having denominator 2 over a compact Riemann surface $X$ of genus $g$ with a divisor of $s$-many distinct points on $X$, such that $2g-2+s>0$, has $(2^s+1)2^{2g+s-1}+(2g-3+s)2^{s}$ many connected components.
\end{thm}

The proof that the analogous subspaces of ${{\mathsf{\mathcal{M}}}_{par}^{max}}\left( \text{Sp}\left( 2n,\mathbb{R} \right) \right)$ for $n\ge 3$ are connected is similar, thus providing the following:

\begin{thm}\label{510}
For $n \ge 3$, the moduli space  ${{\mathsf{\mathcal{M}}}_{par}^{\text{max}}}\left( \text{Sp}\left( 2n,\mathbb{R} \right) \right)$  of parabolic maximal $\text{Sp}\left( 2n,\mathbb{R} \right)$-Higgs bundles with all weights rational having denominator 2 over a compact Riemann surface $X$ of genus $g$ with a divisor of $s$-many distinct points on $X$, such that $2g-2+s>0$, has $(2^s+1)2^{2g+s-1}$ many connected components.
\end{thm}

\section{Connected Components of $\mathcal{M}^{max}_{par}(\text{Sp}\left( 2n,\mathbb{R} \right))$ with Fixed Weight $\frac{1}{2}$}
In this section we further restrict to the moduli space $\mathcal{M}^{max}_{par}(\text{Sp}\left( 2n,\mathbb{R} \right), \hat{\alpha})$ of maximal parabolic $\text{Sp}(2n,\mathbb{R})$-Higgs bundles $(V,\beta,\gamma)$ for a fixed parabolic structure $\hat{\alpha}$ on the parabolic bundle $V$ described by a trivial flag over each point $x\in D$ and weight $\frac{1}{2}$.

The maximality of $(V,\beta,\gamma)$ defines a Cayley partner $(W,c,\phi)$ with $W$ having also a fixed parabolic structure induced by the one on $V$. We now notice that among the three apparent cases $(1),(2a)$ and $(2b)$ from \S 6 there is a difference appearing only in Case $(2a)$. In this case $W = L \oplus L^{\vee}$, however the parabolic degree of $L$ is now also fixed satisfying
\begin{align*}
0 \leq par\deg(L) < 2g-2+s.
\end{align*}
The parabolic degree $par\deg(L)$ defines a topological invariant and with a similar proof as in Proposition \ref{503}, we can prove that the corresponding component of the moduli space ${{\mathsf{\mathcal{M}}}_{par}^{\text{max}}}\left( \text{Sp}\left( 4,\mathbb{R} \right), \hat{\alpha} \right)$ is connected. Therefore, we have the following corollary which verifies the prediction for the exact number of connected components from \cite{KySZ}:

\begin{cor}\label{801}
The number of connected components of the moduli space ${{\mathsf{\mathcal{M}}}_{par}^{\text{max}}}\left( \text{Sp}\left( 4,\mathbb{R} \right), \hat{\alpha} \right)$ with fixed trivial filtration and weight $\frac{1}{2}$ is
\begin{align*}
2^{2g+s-1}(2^s+1)+(2g-2+s)-2^s.
\end{align*}
\end{cor}

\begin{rem}\label{802}
The number of connected components is the same for any parabolic structure $\alpha$ in which all the weights are rational numbers with denominator 2, in other words, the above component count does not depend on the fixed filtration of the parabolic bundle.
\end{rem}

\section{The Case of rational Weights}

We can analogously consider maximal $\text{Sp}(2n,\mathbb{R})$-Higgs bundles for a more general choice of weights.
\begin{defn}\label{901}
The parabolic $\text{Sp}(2n,\mathbb{R})$-Higgs bundle $(V,\beta,\gamma)$ with \textit{weight type of $(m_i)_{1\le i\le s}$} is defined to be the parabolic Higgs bundle with $(\beta,\gamma)$ as in Definition \ref{401}, and the weight $\alpha_{i,j}$ at each point $x_i\in D$ in the parabolic structure of $V$ is an integral multiple of $\frac{1}{m_i}$ for $1 \leq i \leq s$, $1 \leq j \leq n$.
\end{defn}
Since the proof for maximality of the Toledo invariant does not depend on the parabolic structure, we still have $|par\deg V|\le n\left(g-1+\frac{s}{2}\right)$. So when $par\deg V=n\left(g-1+\frac{s}{2}\right)$, we may still call $V$ the \textit{maximal} parabolic $\text{Sp}(2n,\mathbb{R})$-Higgs bundle with weight type of $(m_i)_{1\le i\le s}$. The moduli space of polystable maximal parabolic $\text{Sp}(2n,\mathbb{R})$-Higgs bundles with weight type of $(m_i)_{1\le i\le s}$ is denoted by
$$\mathcal{M}_{par,(m_1,\ldots,m_s)}^{\text{max}}(\text{Sp}(2n,\mathbb{R})).$$

With respect to the data $(X,D,(m_i)_{1 \leq i \leq s})$, we can construct the $V$-surface $M$, which is given by gluing two charts $U_1$ and $U_2$, where $U_1= X \backslash \{x_1,...,x_s\}$ and $U_2 = \coprod_{i=1}^s D / \mathbb{Z}_{m_i}$. Then, we define $M_V$ as follows
\begin{align*}
M_V=V_1 \bigcup V_2, \quad V_1 = X \backslash \{x_1,...,x_s\}, \quad V_2 = \coprod_{i=1}^s D\times_{\mathbb{Z}_m} E\mathbb{Z}_m.
\end{align*}
This construction is similar to what we did in \S 6, and the only thing we change is the weight $m_i$.

By the Mayer-Vietoris sequence, we have
\begin{align*}
0 \rightarrow H^0(M_V,\mathbb{Z}_2) \rightarrow H^0(V_1,\mathbb{Z}_2) \bigoplus H^{0}(V_2,\mathbb{Z}_2) \rightarrow H^0(V_1\bigcap V_2,\mathbb{Z}_2)\\
\xrightarrow[]{j_1}  H^1(M_V,\mathbb{Z}_2) \rightarrow H^1(V_1,\mathbb{Z}_2) \bigoplus H^{1}(V_2,\mathbb{Z}_2) \rightarrow H^1(V_1\bigcap V_2,\mathbb{Z}_2)\\
\xrightarrow[]{j_2} H^2(M_V,\mathbb{Z}_2) \rightarrow H^2(V_1,\mathbb{Z}_2) \bigoplus H^{2}(V_2,\mathbb{Z}_2) \rightarrow H^2(V_1\bigcap V_2,\mathbb{Z}_2).\\
\end{align*}
Note that $V_1=X \backslash \{x_1,...,x_s\}$ and $V_1\bigcap V_2=\prod_{i=1}^s S^1$. We have
\begin{align*}
& \text{rk}(H^0(V_1\bigcap V_2))=s, \quad \text{rk}(H^0(V_1))=1, \\
& \text{rk}(H^1(V_1\bigcap V_2))=s, \quad \text{rk}(H^1(V_1))=2g+s-1, \\
& \text{rk}(H^2(V_1\bigcap V_2))=0, \quad \text{rk}(H^2(V_1))=0.
\end{align*}

Now, if we want to calculate the $\mathbb{Z}_2$-cohomology group of $M_V$, we have to figure out the cohomology group $H^i(V_2,\mathbb{Z}_2)$ for $0 \leq i \leq 2$, which is equivalent to calculate the $\mathbb{Z}_2$-cohomology group of $D\times_{\mathbb{Z}_m}E\mathbb{Z}_m$. We use the Leray spectral sequence and the group cohomology of $\mathbb{Z}_m$ to calculate $H^i(D\times_{\mathbb{Z}_m}E\mathbb{Z}_m,\mathbb{Z}_2)$. We have
\begin{align*}
H^0(D\times_{\mathbb{Z}_m}E\mathbb{Z}_m,\mathbb{Z}_2)=\mathbb{Z}_2,
\end{align*}

\begin{align*}
H^1(D\times_{\mathbb{Z}_m}E\mathbb{Z}_m,\mathbb{Z}_2)=m-\mbox{torsion submodule of }\mathbb{Z}_2= \left\{
\begin{array}{ll}\mathbb{Z}_2 &\mbox{ if }m\mbox{ is even}\\ 0 & \mbox{ if }m\mbox{ is odd} \end{array}
\right .
\end{align*}

\begin{align*}
H^2(D\times_{\mathbb{Z}_m}E\mathbb{Z}_m,\mathbb{Z}_2)=\mathbb{Z}_2/m\mathbb{Z}_2= \left\{
\begin{array}{ll}\mathbb{Z}_2 &\mbox{ if }m\mbox{ is even}\\ 0 & \mbox{ if }m\mbox{ is odd} \end{array}
\right.
\end{align*}

\begin{rem}
Recall that, in \S 6, the coefficient in the group cohomology is $\mathbb{Z}_2$, which coincides with the group $\mathbb{Z}_2$ acting around punctures in the orbifold $M$. Indeed, the reason why we take $\mathbb{Z}_2$ as the coefficient in the group cohomology comes from the discussion of the various cases in \S 6. We consider the real structure of the bundle $W$, and then take the Stiefel-Whitney classes as the topological invariants of $W$. On the other hand, the $\mathbb{Z}_2$-action around punctures comes from the assumption that all weights in the parabolic structure can be written as a fraction with denominator $2$ (see Definition \ref{401} and \S 6).

Now we come back to the setup in this section. In Definition \ref{901}, the weights around the puncture $x_i$ can be written as a fraction with denominator $m_i$. This property determines the group $\mathbb{Z}_{m_i}$, which is acting on the puncture in the corresponding orbifold. On the other hand, the coefficient $\mathbb{Z}_2$ in the cohomology still comes from a similar discussion as in \S 6 (see Case (1)).
\end{rem}

Denote by $s_0$ the number of the even $m_i$'s in the collection $(m_i)_{1\le i\le s}$ and by $s_1$ the number of the odd $m_i$'s, so clearly $s=s_0+s_1$. From a similar Mayer-Vietoris sequence as the one considered in \S 7 of \cite{KySZ}, it follows that the corresponding $V$-cohomology for the corresponding $V$-surface $M_{V}$ has rank
\begin{align*}
\text{rk}(H^1(M_V,\mathbb{Z}_2))=2g+s_0-1;\quad \text{rk}(H^2(M_V,\mathbb{Z}_2))=s_0.
\end{align*}
Topological invariants for elements in $\mathcal{M}_{par,(m_1,\ldots,m_s)}^{\text{max}}(\text{Sp}(2n,\mathbb{R}))$ can be similarly obtained, simply replacing $s$ by $s_0$ in \S 6.

For the case of $\mathcal{M}_{par,(m_1,\ldots,m_s)}^{\text{max}}(\text{Sp}(4,\mathbb{R}))$:
\begin{enumerate}
\item If $\wedge^2 W\not=0$, the structure group is $\text{O}(2,\mathbb{C})$ and we still have $0\not=u\in H^1_V(M,\mathbb{Z}_2)$ and $v\in H^2_V(M,\mathbb{Z}_2)$, thus giving $2^{s_0}(2^{2g+s_0-1}-1)$ many different pairs $(u,v)$. Denote by $\mathcal{M}_{par,(m_1,\ldots,m_s)}^{(u,v)}$ the moduli space of pairs with fixed $(u,v)$.
\item If $\wedge^2 W=0$, then we still have $W=L\oplus L^{\vee}$ and $par\deg(L)\le 2g-2+s$.
\begin{enumerate}
\item If $par\deg(L)\not= 2g-2+s$, we need to describe all possible positive degrees less than $2g-2+s$. Note that now the exact sequence of the Picard V-group is
    \[0\to \text{Pic}(X)\to \text{Pic}_V(M)\to \bigoplus_{i=1}^s \mathbb{Z}_{m_i}\to 0,\]
while for each choice of an element $(w_i)\in \bigoplus_{i=1}^s \mathbb{Z}_{m_i}$ we have an integral degree $d$ such that $d+\sum \frac{w_i}{m_i}<2g-2+s$. These form the moduli space $\mathcal{M}_{par,(m_1,\ldots,m_s)}^{0,(d,w)}$.
\item If $par\deg L=2g-2+s$, then $L^2=K(D)^2$, and the number of square roots of $K(D)$ is similarly the rank of 2-torsion in $H^1(M_V,\mathbb{Z})$, which is $2^{2g+s_0-1}$ as expected. Denote the pairs in this case by $\mathcal{M}_{par,(m_1,\ldots,m_s)}^{0,2g-2+s_0,L}$.\\
\end{enumerate}
\end{enumerate}
For the case of $\mathcal{M}_{par,(m_1,\ldots,m_s)}^{\text{max}}(\text{Sp}(2n,\mathbb{R}))$, we similarly get:
\begin{enumerate}
\item If $\beta=0$, then our moduli spaces are parametrized by $u\in H_V^1(M,\mathbb{Z}_2)$ and $v\in H_V^2(M,\mathbb{Z}_2)$.
\item If $\beta\not=0$ then there are $2^{2g+s_0-1}$ choices of square roots of $K(D)$.
\end{enumerate}
The only case for which the weight with odd monodromy does contribute is that of $\text{Sp}(4,\mathbb{R})$, when $W=L\oplus L^{\vee}$, with $\deg L\not= 2g-2+s$. In this situation, we can first choose the number $(k_1,\ldots, k_s)$, and $d$ must satisfy the inequality
\[-\sum_{i=1}^s\frac{k_i}{m_i}\le d < 2g-1+s-\sum_{i=1}^s\frac{k_i}{m_i}.\]
Therefore, there are $2g-1+s$ choices for $d$, for any choice of $(k_1,\ldots,k_s)$. The total number of such components is thus
$$(2g-1+s)\prod_{i=1}^s m_i=(2g-1+s_0+s_1)\prod_{i=1}^s m_i.$$

From Propositions \ref{501}, \ref{503}, \ref{505}, the components $\mathcal{M}_{par,(m_1,\ldots,m_s)}^{(u,v)}$, $\mathcal{M}_{par,(m_1,\ldots,m_s)}^{0,(d,w)}$, $\mathcal{M}_{par,(m_1,\ldots,m_s)}^{0,2g-2+s_0,L}$ are all connected.

In conclusion, we have the following theorems

\begin{thm}\label{903}
Let $X$ be a compact Riemann surface of genus $g$ with a divisor $D$ of $s$-many distinct points on $X$, such that $2g-2+s>0$. Then, the moduli space of poly-stable maximal parabolic $\text{Sp}(4,\mathbb{R})$-Higgs bundles on $(X,D)$ with weight type $(m_i)_{1\le i\le s}$ has $$(2^{s_0}+1)2^{2g+s_0-1}-2^{s_0}+(2g-2+s)\prod_{i=1}^s m_i$$ connected components, where $s_0$ is the number of the even $m_{i}$ in the collection $(m_i)_{1\le i\le s}$.
\end{thm}
The case when $n \ge 3$ is simpler:
\begin{thm}\label{904}
Let $X$ be a compact Riemann surface of genus $g$ with a divisor $D$ of $s$-many distinct points on $X$, such that $2g-2+s>0$. Then, for $n \ge 3$ the moduli space of poly-stable maximal parabolic $\text{Sp}(2n,\mathbb{R})$-Higgs bundles on $(X,D)$ with weight type $(m_i)_{1\le i\le s}$ has $$(2^{s_0}+1)2^{2g+s_0-1}$$ connected components, where $s_0$ is the number of the even $m_{i}$ in the collection $(m_i)_{1\le i\le s}$.
\end{thm}

Before we move on to the next Corollary, we introduce the \emph{reduced} and \emph{non-reduced parabolic structure} $\alpha$ with respect to the given weight type $(m_i)_{1\le i\le s}$. By definition, the parabolic structure is uniquely determined by the weights over each puncture $x \in D$, thus let $\alpha=(\alpha_{ij})_{1 \leq i \leq s,1 \leq j \leq n}$, where $(\alpha_{ij})_{1 \leq j \leq n}$ is the set of weights over the puncture $x_i$ with the same denominator $m_i$. Let $\alpha_{ij}=\frac{k_{ij}}{m_i}$, where $k_{ij}$ is a positive integer. For any $1 \leq i \leq s$, if there is at least one enumerator $k_{ij}$ such that $k_{ij}$ and $m_i$ are co-prime, the parabolic structure $\alpha$ is called \emph{non-reduced}. Otherwise, the parabolic structure $\alpha$ is called \emph{reducible}.

By using the same arguments as in \S 8, we also obtain
\begin{cor}\label{905}
Let $X$ be a compact Riemann surface of genus $g$ with a divisor $D$ of $s$-many distinct points on $X$, such that $2g-2+s>0$. Then, the moduli space of polystable maximal parabolic $\text{Sp}(4,\mathbb{R})$-Higgs bundles on $(X,D)$ with any non-reduced parabolic structure $\alpha$ has $$(2^{s_0}+1)2^{2g+s_0-1}-2^{s_0}+(2g-2+s)$$ connected components, where $s_0$ is the number of the even $m_{i}$ in the collection $(m_i)_{1\le i\le s}$.
\end{cor}

To deal with the reducible case, we can turn a non-reduced parabolic structure into a reducible one by fraction reduction. Then, this goes back to the non-reduced case as in Corollary \ref{905}.

\vspace{2mm}
\textbf{Acknowledgments}.
The authors wish to express their warmest acknowledgements to the anonymous referee for their valuable suggestions on this article. We are also very grateful to Marina Logares for providing useful comments on a preliminary version of this article. G. K. kindly thanks the Labex IRMIA of the Universit\'{e} de Strasbourg for  support during the completion of this project.

\bigskip

\noindent\small{\textsc{Institut de Recherche Math\'{e}matique Avanc\'{e}e, Universit\'{e} de Strasbourg}\\
7 rue Ren\'{e}-Descartes, 67084 Strasbourg Cedex, France}\\
\emph{E-mail address}:  \texttt{kydonakis@math.unistra.fr}

\bigskip
\noindent\small{\textsc{Department of Mathematics, Sun Yat-Sen University}\\
135 Xingang W Rd, BinJiang Lu, Haizhu Qu, Guangzhou Shi, Guangdong Sheng, China}\\
\emph{E-mail address}:  \texttt{sunh66@mail.sysu.edu.cn}

\bigskip
\noindent\small{\textsc{Department of Mathematics, University of Illinois at Urbana-Champaign}\\
1409 W. Green St, Urbana, IL 61801, USA}\\
\emph{E-mail address}: \texttt{lzhao35@illinois.edu}
\end{document}